\documentclass[12pt]{amsart}
\usepackage{color,amsmath,amsfonts,amssymb,amscd,amsthm,amsbsy,upref,hyperref}
\numberwithin{equation}{section}
\def\hangbox to #1 #2{\vskip3pt\hangindent #1\noindent \hbox to #1{#2}$\!\!$}

\newtheorem{thm}{Theorem}[section]

\newtheorem{lem}[thm]{Lemma}
\newtheorem{cor}[thm]{Corollary}

\theoremstyle{definition}

\newtheorem{defn}[thm]{Definition}

\theoremstyle{remark}

\def\N{{\mathbb N}}
\def\R{{\mathbb R}}

\def\E{{\mathbb E}}

\def\sfrac#1#2{\kern.1em\raise.5ex\hbox{$#1$}
        \kern-.1em/\kern-.05em\lower.25ex\hbox{$#2$}}

\def\vp{\varepsilon}

\makeatletter
\newcommand{\pushright}[1]{\ifmeasuring@#1\else\omit\hfill$\displaystyle#1$\fi\ignorespaces}
\newcommand{\pushleft}[1]{\ifmeasuring@#1\else\omit$\displaystyle#1$\hfill\fi\ignorespaces}
\makeatother

\newcommand{\fw}{\text{\fw}}

\begin{document}
\allowdisplaybreaks

\title{Stable phase retrieval for infinite dimensional subspaces of $L_2(\R)$}

 \author{Robert Calderbank}
 \address{Department of Electrical and Computer Engineering\\
   Duke University\\
    Durham, NC USA}\email{Robert.Calderbank@math.duke.edu}
 
 \author{ Ingrid Daubechies}
 \address{Department of Mathematics and Department of Electrical and Computer Engineering\\
   Duke University\\
    Durham, NC USA}\email{ingrid@math.duke.edu}

 \author{ Daniel Freeman}
\address{Department of Mathematics and Statistics\\
St Louis University\\
St Louis, MO   USA} \email{daniel.freeman@slu.edu}

  \author{ Nikki Freeman}
\address{Department of Biostatistics\\
University of North Carolina at Chapel Hill\\
Chapel Hill, NC   USA} \email{nlbf@live.unc.edu}

\thanks{2020 \textit{Mathematics Subject Classification}: 42C15, 46B42, 94A20
}

\thanks{The second author was supported by grant 400837 from the Simons Foundation.  The third author was supported by grant 706481 from the Simons Foundation.}
 
 \begin{abstract}
 Phase retrieval is known to always be unstable when using a frame or continuous frame for an infinite dimensional Hilbert space.  We consider a generalization of phase retrieval to the setting of subspaces of $L_2$ which coincides with using a continuous frame for phase retrieval when the subspace is the range of the analysis operator of a continuous frame.  We then prove that there do exist infinite dimensional subspaces of $L_2$ where phase retrieval is stable.  That is, we give a method for constructing an infinite dimensional subspace $Y\subseteq L_2$ such that there exists $C\geq 1$ so that 
$$\min\big(\big\|f-g\big\|_{L_2},\big\|f+g\big\|_{L_2}\big)\leq C \big\| |f|-|g| \big\|_{L_2} \qquad\textrm{ for all }f,g\in Y.
$$

This construction also leads to new results on uniform stability of phase retrieval in finite dimensions.  Our construction has a deterministic component and a random component. 
When using sub-Gaussian random variables we achieve phase retrieval with high probability and stability constant independent of the dimension $n$ when using $m$ on the order of $n$ random vectors.  Without sub-Gaussian or any other higher moment  assumptions, 
we are able to achieve phase retrieval with high probability and stability constant independent of the dimension $n$ when using $m$ on the order of $n\log(n)$ random vectors.  
\end{abstract}

 \maketitle
 \section{Introduction}

A frame $(\phi_j)_{j\in I} \subseteq H$ for a Hilbert space $H$ allows for any vector $x\in H$ to be linearly recovered from the collection of frame coefficients $(\langle x, \phi_j\rangle)_{n\in I}$.  
 However, there are many instances in physics and engineering where one is able to obtain only the magnitude of linear measurements.  Notable examples occur in speech recognition \cite{BR}, coherent diffraction imaging \cite{MCKS}, X-ray crystallography \cite{T}, and transmission electron microscopy \cite{K}. In such
cases, one must use phase retrieval to reconstruct a signal.  Given $x\in H$, the goal of phase retrieval is to determine $x$ (up to a unimodular scalar) from the collection of values $(|\langle x, \phi_j\rangle|)_{j\in I}$.  
 When $H$ is a real Hilbert space we say that $(\phi_n)_{n\in I}$ {\em does phase retrieval } if whenever $x,y\in H$ and $(|\langle x, \phi_j\rangle|)_{j\in I}=(|\langle y, \phi_j\rangle|)_{j\in I}$ we have that $x=y$ or $x=-y$.  
 Though phase retrieval may be considered for both real and complex Hilbert spaces, we will only be considering real Hilbert spaces in this paper.
 
When $(\phi_j)_{j\in I}\subseteq H$ is a frame of $H$, we have that  the analysis operator $T:H\rightarrow \ell_2(I)$ is an embedding of $H$ into $\ell_2(I)$, where  $T(x):=(\langle x, \phi_j\rangle)_{j\in I}$ for all $x\in H$.    
Thus, being able to recover  each $x\in H$ (up to a unimodular scalar) from the collection of values $(|\langle x, \phi_n\rangle|)_{n\in I}$ is equivalent to being able to recover  each $f=(f_j)_{j\in I}\in T(H)\subseteq \ell_2(I)$ (up to a unimodular scalar) from its absolute value $|f|=(|f_j|)_{j\in I}\in\ell_2(I)$.  This equivalent formulation leads us to the definition that {\em a subspace $Y\subseteq\ell_2(I)$ does  phase retrieval} if for all $f,g\in Y\subseteq\ell_2(I)$ with $|f|=|g|$ we have that $f=g$ or $f=-g$.  In particular, a frame $(\phi_j)_{j\in I}$  of $H$ with analysis operator $T$ does phase retrieval if and only if the subspace $T(H)\subseteq \ell_2(I)$ does phase retrieval.

We think of phase retrieval for a subspace as being {\em stable} if for all $f,g\in Y\subseteq \ell_2(I)$ when $|f|$ is close to $|g|$ then $f$ is proportionally close to either $g$ or $-g$.  That is, {\em $Y\subseteq \ell_2(I)$ does stable phase retrieval } if there exists a constant $C>0$ such that for all $f,g\in Y$ we have that $\min(\|f-g\|,\|f+g\|)\leq C \||f|-|g|\|$.  If we consider the equivalence relation $\sim$ on $Y$ to be $f\sim g$ if and only if $f=g$ or $f=-g$ then $Y\subseteq \ell_2(I)$ does $C$-stable phase retrieval is equivalent to the map $|f|\mapsto f/\sim$ is well defined and is $C$-Lipschitz.  Because of this, having a good stability bound for phase retrieval is of fundamental importance in applications.

Phase retrieval using a frame for a finite-dimensional Hilbert space is always stable \cite{B}\cite{CCPW}, although the stability constant may be very large. Thus for all $n\in\N$, if a subspace $Y\subseteq \ell_2^n$ does phase retrieval then $Y\subseteq \ell_2^n$ does stable phase retrieval.  In contrast to this, phase retrieval using a frame for an infinite-dimensional Hilbert space is always unstable \cite{CCD}.  Thus, if $Y\subseteq\ell_2$ and $Y$ is infinite-dimensional then $Y$ does not do stable phase retrieval.  The basic reason for this instability is that if $Y\subseteq \ell_2$ is infinite-dimensional then for all $\vp>0$ there are always vectors $x,y\in \ell_2$ with $\|x\|=\|y\|=1$ such that $x$ and $y$ may be slightly perturbed to have disjoint support.  In particular, we can choose $x,y\in Y$ so that almost all of $x$ comes before almost all of $y$ and so there is a sequence $(a_1,a_2,...)\in\ell_2$ with 
 $$\|x-(a_1,a_2,...,a_k,0,0,0,...)\|<\vp \textrm{ and }\|y-(0,...,0,a_{k+1},a_{k+2},...)\|<\vp.$$
Thus, if $f=x+y$ and $g=x-y$ then $|f|$ and $|g|$ will both be approximately $(|a_1|,|a_2|,|a_3|,...)$ so $|f|$ and $|g|$ will be very close.     However, $\|f-g\|=2\|y\|=2$ and $\|f-(-g)\|=2\|x\|=2$.  So even though $|f|$ and $|g|$ are close, $f$ is 2 away from $g$ and $f$ is 2 away from $-g$.  Thus phase retrieval for infinite dimensional subspaces of $\ell_2$ is not stable.

    We now consider phase retrieval for infinite dimensional subspaces $Y\subseteq L_2(\R)$ instead of $Y\subseteq \ell_2$, where we use $L_2(\R)$ to denote the space of square integrable real valued functions defined on $\R$.  In particular, we will give a method for constructing an infinite-dimensional closed subspace $Y\subseteq L_2(\R)$ and constant $C>0$ such that for all $f,g\in Y$ we have that $\min(\|f-g\|_{L_2},\|f+g\|_{L_2})\leq C \||f|-|g|\|_{L_2}$.  That is, unlike in $\ell_2$, stable phase retrieval is possible for  infinite-dimensional subspaces of $L_2(\R)$.
    
The idea of doing phase retrieval in $L_2(\R)$ was considered in \cite{AG}, where the authors study phase retrieval using continuous frames for both Hilbert spaces and general Banach spaces.  If $(\phi_t)_{t\in\R}$ is a continuous frame of a separable Hilbert space $H$, then    the analysis operator $T:H\rightarrow L_2(\R)$ given by $(T(f))(t)=\langle f, \phi_t\rangle$ for all $t\in\R$ is an isomorphic embedding of $H$ into $L_2(\R)$.  Thus phase retrieval for a continuous frame $(\phi_t)_{t\in\R}$ of $H$ is equivalent to phase retrieval for the subspace $T(H)\subseteq L_2(\R)$.  Like in the case for discrete frames,  phase retrieval using a continuous frame for an infinite-dimensional Hilbert space is always unstable \cite{AG}.  However, this does not imply that phase retrieval for every infinite-dimensional subspace of $L_2(\R)$ is unstable.  This is where one of the key differences between discrete and continuous frames comes in.  Every closed subspace of  $\ell_2$ is the range of the analysis operator of a discrete frame $(\phi_n)_{n=1}^\infty$.  On the other hand, not every closed subspace of $L_2(\R)$ is the range of the analysis operator of a continuous frame $(\phi_t)_{t\in\R}$.   Thus, all of the subspaces $Y\subseteq L_2(\R)$ for which we prove do stable phase retrieval must necessarily  not be the range of the analysis operator of a continuous frame.  
   
   Many of our proofs will rely on probabilistic arguments.  As we hope that our paper will be accessible to a broad mathematical community we recall now some definitions and notation.  For a measurable set $E\subseteq[0,1]$ we use $\textrm{Prob}(E)$ to denote Lebesgue measure of $E$. We use the term {\em random variable} to refer to a measurable function $f$ from a probability space to $\R$.  If $f$ is a random variable and $F\subseteq\R$ is measurable we use the notation $\textrm{Prob}(f\in F):=\textrm{Prob}(f^{-1}(F))$.    We say that two random variables $f,g$ {\em have the same distribution} if for all measurable $F\subseteq\R$ we have that $\textrm{Prob}(f\in F)=\textrm{Prob}(g\in F)$.
     If $(y_n)_{n=1}^\infty$ is a sequence in $L_2([0,1])$ then we say that $(y_n)_{n=1}^\infty$ is {\em independent} if for all $J\subseteq\N$ and $E,F\subseteq\R$ we have that if $x\in span(y_n)_{n\in J}$ and $y\in span(y_n)_{n\in J^c}$ then $\textrm{Prob}(x\in E\textrm{ and }y\in F)=\textrm{Prob}(x\in E)\textrm{Prob}(y\in F)$.

 We will be considering subspaces of $L_2(\R)$ of the form $\overline{span}(y_j+1_{(j,j+1]})$ where $(y_j)_{j=1}^\infty$ is an orthonormal sequence of independent mean-zero random variables in $L_2([0,1])$ and $(1_{(j,j+1]})_{j=1}^\infty$ is just a sequence of indicator functions of disjoint intervals.  The following theorem gives multiple characterizations for when a subspace of that form does stable phase retrieval in $L_2(\R)$.
 
\begin{thm}\label{T:main}
Let $(y_j)_{j=1}^\infty$ be an orthonormal sequence of independent mean-zero random variables in $L_2([0,1])$.  Then the following are equivalent,
\begin{enumerate}
\item The subspace $\overline{span} (y_j+1_{(j,j+1]})_{j=1}^\infty$ does stable phase retrieval in $L_2(\R)$.
\item There exist constants $a,\gamma>0$ with $\textrm{Prob}(|x|\geq a\|x\|)\geq\gamma$ for all $x\in \overline{span}(y_j)_{j=1}^\infty$.
\item  The $L_1$ and $L_2$ norms of $y_j$ are uniformly comparable.  That is, there exists $A>0$ such that  $A\|y_j\|_{L_2([0,1]}\leq \|y_j\|_{L_1([0,1])}\leq  \|y_j\|_{L_2([0,1])}$ for all $j\in\N$.
\item It is not the case that for all $\vp>0$, the subspace $\overline{span}(y_j)_{j=1}^\infty\subseteq L_2([0,1])$ contains two independent unit vectors which may be perturbed by less than $\vp$ to have disjoint support.  
\end{enumerate}
Furthermore, if $a,\gamma>0$ satisfy $(2)$ then the subspace $\overline{span} (y_j+1_{(j,j+1]})_{j=1}^\infty\subseteq L_2(\R)$ does $6a^{-1}\gamma^{-1}$-stable phase retrieval.  That is, for all $f,g\in \overline{span} (y_j+1_{(j,j+1]})_{j=1}^\infty$, we have that 
$$min(\| f-g\|_{L_2(\R)},\|f+g\|_{L_2(\R)}) \leq 6a^{-1}\gamma^{-1} \big\| |f|-|g| \big\|_{L_2(\R)}.
$$
\end{thm}

Note that condition (3) is simply that the sequence of norms $(\|y_j\|_{L_1([0,1])})_{j=1}^\infty$ is bounded away from $0$, which is simple to check in most circumstances.  However, checking directly that a subspace $Y\subseteq L_2$ does phase retrieval requires checking a condition for all pairs of vectors $(x,y)\in Y\times Y$.

Recall that the {\em variance} of a mean-zero random  variable $y$ is $\|y\|^2_{L_2}=\int |y|^2$.  If $(y_j)_{j=1}^\infty$ is a sequence of mean-zero independent random variables on $[0,1]$ with finite variance then $(y_j/\|y_j\|_{L_2([0,1])})_{j=1}^\infty$ will be an orthonormal sequence of independent mean-zero random variables in $L_2([0,1])$.  Furthermore, if  the sequence $(y_j)_{j=1}^\infty$ is identically distributed then there exists a constant $A>0$ such that $\|y_j\|_{L_1([0,1])}=A$ for all $j\in\N$.  Hence we have the following corollary due to $(3)\Rightarrow(1)$ in Theorem \ref{T:main}.
\begin{cor}\label{C:main0}
Let $(y_j)_{j=1}^\infty$ be a sequence of mean-zero independent identically-distributed random variables on $[0,1]$ with finite variance.  Then the subspace $\overline{span} (y_j+1_{(j,j+1]})_{j=1}^\infty$ does stable phase retrieval in $L_2(\R)$.
\end{cor}

Note that if we did not add on the indicator functions $(1_{(j,j+1]})_{j=1}^\infty$ then Theorem \ref{T:main} and Corollary \ref{C:main0} would be false.  Indeed, one could take the sequence $(y_j)_{j=1}^\infty$ to be the Rademacher functions, which is an independent sequence of mean-zero, $\pm1$ random variables.  In this case $|y_j|=1$ for all $j\in\N$ and hence $\overline{span}(y_j)_{j=1}^\infty\subseteq L_2(\R)$ does not do phase retrieval.

The full proof of Theorem \ref{T:main} will be broken up into several parts.
We prove $(2)\Rightarrow (1)$ and the furthermore  part of Theorem \ref{T:main} in Section \ref{S:2}.  We prove the remaining parts of Theorem \ref{T:main}  in Section \ref{S:3} along with some lemmas comparing different $L_p$-norms.   
The equivalences $(2)\Leftrightarrow(3)\Leftrightarrow(4)$  follow from classical results in Banach lattice theory, and we include complete proofs so that the paper can be self-contained.
 
Our results on stable phase retrieval in infinite dimensions can be used to prove new uniform stability theorems for phase retrieval in finite dimensions. It is well known that phase retrieval for finite dimensional subspaces is always stable.  However, every known deterministic construction of a frame which does phase retrieval either has very high stability constant  or must use a very large number of vectors relative to the dimension of the space.  There have recently been many strong results where random methods are used to construct frames that with high probability perform phase retrieval with stability constant independent of the dimension of the space.
  The first result on uniform stability for phase retrieval in an $n$-dimensional Hilbert space was for  $m$ iid random vectors with uniform distribution on the sphere where $m$ was  on the order of $n\log(n)$ \cite{CSV}. This was then improved to $m$ being on the order of $n$ \cite{CL}.
  Uniform stability for phase retrieval with high probability using $m$ on the order of $n$ random vectors was then proven for many other sub-Gaussian distributions as well, assuming some additional conditions such as small-ball probability assumptions \cite{EM} or lower $L_\infty$ bounds \cite{KL}. 
  In \cite{GKK}, the authors give a partially de-randomized construction by proving that randomly sampling spherical designs results in a frame which does stable phase retrieval with high probability.
  By constructing an infinite dimensional subspace of $L_2(\R)$ which does stable phase retrieval, we have that every finite dimensional subspace does stable phase retrieval with stability constant independent of the dimension.  As every finite dimensional subspace of $L_2(\R)$ is the range of the analysis operator of a continuous frame, we are thus able to provide new constructions for continuous frames of $\ell_2^n$ which do stable phase retrieval for all $n\in\N$ with stability constant independent of the dimension.  We now consider the problem of sampling this continuous frame to obtain frames which do stable phase retrieval.
  Our construction has a random component and a deterministic component.  We prove that for all $n\in\N$ that if the random component consists of sub-Gaussian random variables then we may choose $m$ on the order of $n$ random vectors to construct a frame for $\ell_2^n$ which with high probability does phase retrieval with stability constant independent of the dimension.  
  We define what it means for a random variable to be sub-Gaussian and prove the following theorem in Section \ref{S:4}.  Recall that the analysis operator of a frame $(\phi_j)_{j=1}^N$ for $\ell_2^n$ is the map $T:\ell_2^n\rightarrow \ell_2^N$ given by $T(x)=(\langle x, \phi_j\rangle)_{j=1}^N$ for all $x
  \in\ell_2^N$.
  
  \begin{thm}\label{T:finite_sG}
There exists universal constants $c_1,c_2,c_3>0$ such that the following holds.
Let $(y_j)_{j=1}^n$ be an independent sequence of  mean-zero, variance-one, $K$-sub-Gaussian random variables.  Let $(v_j)_{j=1}^m$ be an independent sequence in $\ell_2^n$ with the same distribution as $v=(y_1,y_2,...,y_n)$. If $m\geq c_1 K^{16}\log(2K) n$ then with probability at least $1-\exp\left(-c_2 m \right)$ the frame $(\frac{1}{\sqrt{m}}v_j)_{j=1}^m\cup (e_i)_{i=1}^n$ does $(c_3 K^{6})$-stable phase retrieval.   That is, if $T:\ell_2^n\rightarrow\ell_2^{m+n}$ is the analysis operator for the frame $(\frac{1}{\sqrt{m}}v_j)_{j=1}^m\cup (e_i)_{i=1}^n$ then for all $f,g\in\ell_2^{n}$ we have that 

$$
min(\| f-g\|_{\ell_2^n},\| f+g\|_{\ell_2^n}) \leq c_3 K^{6}
\big\| |Tf|-|Tg| \big\|_{\ell_2^{m+n}}. 
$$

\end{thm}

 In the following theorem we obtain uniform stability without sub-Gaussian assumptions.  However, the  number of random vectors required is increased to be on the order of $n\log(n)$. This appears to be the first result on uniform stability of phase retrieval which does not rely on any higher moment conditions.

  \begin{thm}\label{T:finite}
Let $a,\gamma>0$.  Then there exist $k_1,k_2>0$ which depend only on the values $a$ and $\gamma$ such that for all $n\in\N$ and $m\geq k_1 n\log(n)$ the following holds.
Suppose that $(y_j)_{j=1}^n$ is a sequence of  mean-zero, variance-one, independent random variables such that  $\textrm{Prob}(|x|\geq a\|x\|)\geq\gamma$ for all $x\in \overline{span}(y_j)_{j=1}^n$.
Let $(v_j)_{j=1}^m$ be an independent sequence of random vectors in $\ell_2^n$ each with the same distribution as $v=(y_1,y_2,...,y_n)$. Then with probability at least $1-\exp\left(-k_2 m \right)$ the frame $(\frac{1}{\sqrt{m}}v_j)_{j=1}^m\cup (e_i)_{i=1}^n$ does $(12a^{-1}\gamma^{-1} +1)$-stable phase retrieval.  That is, if $T:\ell_2^n\rightarrow\ell_2^{m+n}$ is the analysis operator for the frame $(\frac{1}{\sqrt{m}}v_j)_{j=1}^m\cup (e_i)_{i=1}^n$ then for all $f,g\in\ell_2^{n}$ we have that 
$$
min(\| f-g\|_{\ell_2^n},\| f+g\|_{\ell_2^n}) \leq (12a^{-1}\gamma^{-1} +1)
\big\| |Tf|-|Tg| \big\|_{\ell_2^{m+n}}. 
$$
\end{thm}

The condition $\textrm{Prob}(|x|\geq a\|x\|)\geq\gamma$ for all $x\in \overline{span}(y_j)_{j=1}^n$ is useful for calculating stability bounds, but it is difficult to check as one must calculate a probability for every linear combination of $(y_j)_{j=1}^n$.  However, by $(3)\Rightarrow(2)$ in Theorem \ref{T:main}, we have the following corollary with a much simpler hypothesis.

  \begin{cor}\label{C:finite}
Let $A>0$.  Then there exists constants $q_1,q_2,q_3>0$ which depend only on the value $A$ such that for all $n\in\N$ and $m\geq q_1 n\log(n)$ the following holds.
Suppose that $(y_j)_{j=1}^n$ is a sequence of  mean-zero, variance $1$, independent random variables such that $A\leq \|y_j\|_{L_1}$ for all $1\leq j\leq n$.
Let $(v_j)_{j=1}^m$ be an independent sequence of random vectors in $\ell_2^n$ each with the same distribution as $v=(y_1,y_2,...,y_n)$. Then with probability at least $1-\exp\left(-q_2 m \right)$ the frame $(\frac{1}{\sqrt{m}}v_j)_{j=1}^m\cup (e_i)_{i=1}^n$ does $q_3$-stable phase retrieval.  That is, if $T:\ell_2^n\rightarrow\ell_2^{m+n}$ is the analysis operator for the frame $(\frac{1}{\sqrt{m}}v_j)_{j=1}^m\cup (e_i)_{i=1}^n$ then for all $f,g\in\ell_2^{n}$ we have that 
$$
min(\| f-g\|_{\ell_2^n},\| f+g\|_{\ell_2^n}) \leq q_3
\big\| |Tf|-|Tg| \big\|_{\ell_2^{m+n}}. 
$$
\end{cor}

\section{Stable phase retrieval for subspaces of $L_2(\R)$}\label{S:2}

We will be considering  subspaces of $L_2(\R)$ of the form $\overline{span} (y_j+1_{(j,j+1]})_{j=1}^\infty$ where $(y_j)_{j=1}^\infty$ is an orthonormal sequence of independent random variables in $L_2([0,1])$ and $(1_{(j,j+1]})_{j=1}^\infty$ is a sequence of indicator functions of disjoint intervals.  Our main goal is to identify properties of the sequence $(y_j)_{j=1}^\infty$ which determine whether or not the subspace $\overline{span} (y_j+1_{(j,j+1]})_{j=1}^\infty\subseteq L_2(\R)$ does stable phase retrieval in $L_2(\R)$.  The full characterization given in Theorem \ref{T:main} will be proven in multiple steps.  The goal of this section is to prove the following theorem which in particular gives the direction $(2)\Rightarrow (1)$ and the furthermore part of Theorem \ref{T:main}.

\begin{thm}\label{T:main2}
Let $(y_j)_{j=1}^\infty$ be an orthonormal sequence of independent random variables in $L_2([0,1])$ such that there exists  $a,\gamma>0$ with $\textrm{Prob}(|x|\geq a\|x\|_{L_2})\geq\gamma$ for all $x\in \overline{span}(y_j)\subseteq L_2([0,1)]$.    Then for all $f,g\in \overline{span} (y_j+1_{(j,j+1]})$, we have that 
$$min(\| f-g\|_{L_2(\R)},\| f+g\|_{L_2(\R)}) \leq 6a^{-1}\gamma^{-1} \big\| |f|-|g| \big\|_{L_2(\R)}.
$$
Thus, the subspace $ \overline{span} (y_j+1_{(j,j+1]})\subseteq L_2(\R)$ does stable phase retrieval.
\end{thm}

We first prove the following lemma.

\begin{lem}\label{L:signs}
Let $(y_j)_{j=1}^\infty$ be an orthonormal sequence in $L_2([0,1])$. Let  $f=\sum a_j (y_j+1_{(j,j+1]})$ and $g=\sum b_j (y_j+1_{(j,j+1]})$ in $L_2(\R)$.   For all $j\in\N$ we let $\vp_j= sign(a_jb_j)$ and $\delta_j= b_j  - \vp_j a_j$.  Then, 
$$\left\| |\sum a_j y_j| - |\sum \vp_j a_j y_j|\right\|_{L_2}\leq 2\big\| |f|-|g|\big\|_{L_2}\textrm{ and } \left(\sum \delta_j^2\right)^{1/2}\leq \big\| |f|-|g|\big\|_{L_2}.
$$
\end{lem}

\begin{proof}
We first consider the domain $(1,\infty)$.  We have that
\begin{align*}
\big\| |f|-|g|\big\|^2_{L_2(\R)} &\geq \big\| |f|-|g|\big\|^2_{L_2((1,\infty))} \\
&= \left\|\sum |a_j|1_{(j,j+1]}-|b_j|1_{(j,j+1]}\right\|^2_{L_2}\\
&=\sum \left(|a_j|-|b_j|\right)^2\\
&=\sum |\delta_j|^2
\end{align*}
Thus, we have that $(\sum \delta_j^2)^{1/2}\leq \left\| |f|-|g|\right\|_{L_2}$.  We now consider the domain $[0,1]$.  
\begin{align*}
\big\| |f|-|g|\big\|_{L_2(\R)} &\geq \big\| |f|-|g|\big\|_{L_2([0,1])} \\
&=  \left\| |\sum a_j y_j|-|\sum b_j y_j| \right\|_{L_2} \\
&=\left\| |\sum a_j y_j|-|\sum (\vp_j a_j+\delta_j)y_j| \right\|_{L_2}\\
&\geq \left\| |\sum a_j y_j|-|\sum \vp_j a_j y_j| \right\|_{L_2} -\left\|\sum \delta_j y_j \right\|_{L_2}\\
&\geq \left\| |\sum a_j y_j|-|\sum \vp_j a_j y_j| \right\|_{L_2} -\left(\sum \delta_j^2\right)^{1/2}  
\end{align*}
The last inequality gives that
$$\left\| |\sum a_j y_j| - |\sum \vp_j a_j y_j|\right\|_{L_2} \leq \big\| |f|-|g|\big\|_{L_2}+\left(\sum \delta_j^2\right)^{1/2}\leq 2 \big\| |f|-|g|\big\|_{L_2}.
$$
\end{proof}

\begin{lem}\label{L:lower_b}
Let $Y\subseteq L_2([0,1])$ be such that there exists $a,\gamma>0$ so that $\textrm{Prob}(|x|\geq a\|x\|_{L_2})\geq\gamma$ for all $x\in Y$.  
Let $x,y\in Y$ be independent  and $\|x\|\geq \|y\|$.  Then,
$$2a\gamma\|y\|\leq \big\| |x+y|- |x-y| \big\|_{L_2}
$$
\end{lem}

\begin{proof}
We have that 
\begin{align*}
\left\| |x+y|- |x-y| \right\|^2_{L_2} &= \int (2 \min(|x|,|y|))^2\\
&\geq 4 \int_{|x|\geq a\|y\|} ( \min(a\|y\|,|y|))^2\\
&= 4 \textrm{Prob}(|x|\geq a\|y\|)\int ( \min(a\|y\|,|y|))^2\qquad\left(\textrm{as $x$ and $y$ are independent}\right) \\
&\geq 4 \textrm{Prob}(|x|\geq a\|x\|)\int ( \min(a\|y\|,|y|))^2\qquad\left(\textrm{as }\|x\|\geq \|y\|\right) \\
&\geq 4 \textrm{Prob}(|x|\geq a\|x\|)\int_{|y|\geq a\|y\|} ( a\|y\|)^2 \\
&= 4 \textrm{Prob}(|x|\geq a\|x\|) \textrm{Prob}(|y|\geq a\|y\|) ( a\|y\|)^2 \\
&\geq 4 \gamma^2  a ^2 \|y\|^2.
\end{align*}

\end{proof}

We are now prepared to prove that stable phase retrieval is possible for infinite dimensional subspaces of  $L_2(\R)$. 

\begin{proof}[Proof of Theorem \ref{T:main2}]
We have that $(y_j)_{j=1}^\infty$ is an orthonormal sequence of independent random variables in $L_2([0,1])$ such that there exists  $a,\gamma>0$ with $\textrm{Prob}(|x|\geq a\|x\|_{L_2})\geq\gamma$ for all $x\in \overline{span}(y_j)\subseteq L_2([0,1)]$.  By Markov's inequality we have that $a\gamma\|x\|_{L_2}\leq \|x\|_{L_1}$ for all $x\in \overline{span}(y_j)\subseteq L_2([0,1)]$.  Hence it is necessary that $a\gamma\leq 1$.

Let $f,g\in\overline{span}(y_j+1_{(j,j+1]})$ with $f=\sum a_j (y_j+1_{(j,j+1]})$ and $g=\sum b_j (y_j+1_{(j,j+1]})$.   For all $j\in\N$ we let $\vp_j= sign(a_jb_j)$ and $\delta_j= b_j  - \vp_j a_j$.  
We let $J=\{j\in\N\,;\, \vp_j=1\}$, $x=\sum_{j\in J}a_j y_j$, and $y=\sum_{j\in J^c} a_j y_j$.  Note that $\sum a_j y_j= x+y$ and $\sum \vp_j a_j y_j=x-y$. Furthermore, $x$ and $y$ are independent as $(y_j)_{j=1}^\infty$ is a sequence of independent random variables. 

  Without loss of generality, we may assume that $\|x\|\geq \|y\|$.  We will now prove that $\| f-g\|_{L_2(\R)} \leq C \| |f|-|g| \|_{L_2(\R)}$ for $C=6a^{-1}\gamma^{-1}$.  (If we had instead assumed that $\|x\|\leq \|y\|$ then we would prove that $\|f+g\|_{L_2(\R)} \leq C \| |f|-|g| \|_{L_2(\R)}$.)

  We now have that
\begin{align*}
6\|& |f|-|g| \|_{L_2}\geq 2 \left\| |\sum a_j y_j| - |\sum \vp_j a_j y_j|\right\|_{L_2}+2\left(\sum \delta_j^2\right)^{1/2} \qquad\big(\textrm{by Lemma \ref{L:signs}}\big)\\
&=2\left\| |x+y| - |x-y|\right\|_{L_2}+2\left(\sum \delta_j^2\right)^{1/2}\\
&\geq 4a\gamma\| y\|_{L_2}+2\left(\sum \delta_j^2\right)^{1/2}\qquad \big(\textrm{by Lemma \ref{L:lower_b}}\big)\\
&=a\gamma\big(\|(x+y)-(x-y)\|_{L_2}+\|(x+y)-(x-y)\|_{L_2}\big)+2\left(\sum \delta_j^2\right)^{1/2}\\
&= a\gamma\left(\left\| \sum a_j y_j - \sum \vp_j a_j y_j\right\|_{L_2} +\left \| \sum a_j 1_{(j,j+1]} - \sum \vp_j a_j 1_{(j,j+1]}\right\|_{L_2} \right)+2\left(\sum \delta_j^2\right)^{1/2}\\
&\geq a\gamma\left(\left\|  \sum a_j y_j - \sum (\vp_j a_j+\delta_j) y_j\right\|_{L_2}+\left\| \sum a_j 1_{(j,j+1]} - \sum (\vp_j a_j+\delta_j)1_{(j,j+1]}\right\|_{L_2}\right) \qquad \big(\textrm{as  $a\gamma\leq 1$.}\big)\\
&=a\gamma\left(\left\|  \sum a_j y_j - \sum b_j y_j\right\|_{L_2}+\left\| \sum a_j 1_{(j,j+1]} - \sum b_j1_{(j,j+1]}\right\|_{L_2}\right) \\
&\geq a\gamma\left\|  \sum a_j (y_j+1_{(j,j+1]}) - \sum b_j (y_j+1_{(j,j+1]})\right\|_{L_2}\\
&= a\gamma\left\|  f - g\right\|_{L_2}
\end{align*}

Thus, $\| f-g\|_{L_2(\R)} \leq 6a^{-1}\gamma^{-1} \| |f|-|g| \|_{L_2(\R)}.$
\end{proof}

\section{Comparing $L_p$ norms and phase retrieval}\label{S:3}

We have been considering an indepedendent orthonormal sequence  $(y_j)_{j=1}^\infty$ in $L_2([0,1])$ and a subspace $Y:=\overline{span}(y_j) \subseteq L_2([0,1]$.  As $[0,1]$ is a probability space, we have for all $x\in Y$ that $\|x\|_{L_2([0,1])}\geq \|x\|_{L_p([0,1])}$ for $1\leq p< 2$, and $\|x\|_{L_p([0,1])}\geq \|x\|_{L_2([0,1])}$ for  $2<p< \infty$.  In this section we will be considering the situation where for some $2<p<\infty$ there exists $1\leq B<\infty$ such that $\|x\|_{L_p([0,1])}\leq B \|x\|_{L_2([0,1])}$ for all $x\in Y$, or for some $1\leq p<2$ there exists $0<A\leq 1$ such that $A\|x\|_{L_2([0,1])}\leq\|x\|_{L_p([0,1])}$ for all $x\in Y$.

Although our results in this section will be infinite dimensional, comparing different $L_p$ norms has interesting connections with phase retrieval in finite dimensions.
In particular,  if $(x_j)_{j=1}^n$ is a spherical design of index p for a finite dimensional Hilbert space $H$ then one has for a suitable constant $C$ that $(\sum_{j=1}^n |\langle x,x_j\rangle|^2)^{1/2}=C (\sum_{j=1}^n |\langle x,x_j\rangle|^p)^{1/p}$ for all $x\in H$.  Thus, if $T:H \rightarrow \ell_2^n$ is the analysis operator for $(x_j)_{j=1}^n$,  we have that $\|y\|_{\ell_2^n}=C \|y\|_{\ell_p^n}$ for every vector $y\in T(H)$. Furthermore, every subspace of $\ell_p^n$ which is isometric to a Hilbert space can be constructed in this way using a spherical design \cite{LV}.  In \cite{GKK}, the authors show that choosing random subsets of spherical designs gives uniformly stable phase retrieval in finite dimension.

We will give three lemmas which each guarantee that a sequence $(y_j)$ satisfies that there exists 
$a,\gamma>0$ such that 
$\textrm{Prob}(|x|\geq a\|x\|_{L_2})\geq\gamma$ for all $x\in span(y_j)$.  Thus, by Theorem \ref{T:main2}, if $(y_j)$ is
an independent orthonormal sequence in $L_2([0,1])$ which satisfies any of these three lemmas then $\overline{span}(y_j+1_{(j,j+1]})\subseteq L_2(\R)$ does stable phase retrieval in $L_2(\R)$.  

\begin{lem}\label{L:bb}
Let $x\in L_2([0,1])$ be such that there exists $2<p<\infty$ and a constant $B\geq1$ with $\|x\|_{L_p([0,1])}\leq B \|x\|_{L_2([0,1)]}$. 
Let $0<a<1$.  Then for $\gamma=B^{2p/(2-p)}(1-a^2)^{p/(p-2)}$ we have that 
$$\textrm{Prob}(|x|\geq a\|x\|_{L_2})\geq\gamma.
$$
\end{lem}
\begin{proof}
Let $0<a<1$ and $x\in Y$.
By scaling, we may assume that $\|x\|_{L_2}=1$.   Let $r=p2^{-1}$.  Note that $1< r< p$ as $2<p$.  Let $r'$ be the dual exponent to $r$.  That is, $1/r+1/r'=1$.
We now calculate, 

\begin{align*}
1=\|x\|_{L_2}^2
&= \left(\int_{|x|\geq a} |x|^2\right)+\left(\int_{|x|< a} |x|^2\right)\\
&\leq \left(\int 1_{|x|\geq a} |x|^2\right)+ a^2\\
&\leq \left(\int 1_{|x|\geq a}\right)^{1/r'}\left(\int |x|^{2r}\right)^{1/r}+a^2\qquad\big(\textrm{by Hölder's inequality}\big)\\
&= \textrm{Prob}(|x|\geq a)^{1/r'}\left(\int |x|^{p}\right)^{2/p}+  a^2\\
&= \textrm{Prob}(|x|\geq a)^{1-2/p}\|x\|^2_{L_p}+ a^2\\
&\leq \textrm{Prob}(|x|\geq a)^{1-2/p}B^2+a^2
\end{align*}

As the exponent $1-2/p$ is positive, solving for $\textrm{Prob}(|x|\geq a)$ gives,
$$\textrm{Prob}(|x|\geq a)\geq (B^{-2}(1-a^2))^{1/(1-2/p)}=B^{2p/(2-p)}(1-a^2)^{p/(p-2)}.
$$

Thus, we have proven the lemma  for $\gamma=B^{2p/(2-p)}(1-a^2)^{p/(p-2)}$ where  $2<p<\infty$.

\end{proof}

In Lemma \ref{L:bb} we used $2<p<\infty$.  We now consider the case $1\leq p<2$.  The proof of the following Lemma is the same as that of Lemma \ref{L:bb} except that we use Hölder's inequality for $r=2p^{-1}$ instead of $r=p2^{-1}$.   The idea of Lemma \ref{L:bb} is that vectors can't be ``too peaky'' because then the $L_p$-norm would be too big for $p>2$, and in the following lemma we prove that vectors can't be ``too peaky'' because then the $L_p$-norm would be too small for $1\leq p<2$. 

\begin{lem}\label{L:bb2}
Let $x\in L_p([0,1])$ be such that there exists $1\leq p<2$ and a constant $A\leq1$ with $\|x\|_{L_p([0,1])}\geq A \|x\|_{L_2([0,1)]}$.  
Let $0<a<A$.  Then for $\gamma=(A^{p}-a^p)^{2/(2-p)}$ we have that
$$\textrm{Prob}(|x|\geq a\|x\|_{L_2})\geq\gamma.
$$
\end{lem}
\begin{proof}
By scaling, we may assume that $\|x\|_{L_2}=1$.   Let $r=2p^{-1}$.  Note that $1< r\leq 2$ as $1\leq p<2$.  Let $r'$ be the dual exponent to $r$.  That is, $1/r+1/r'=1$.
We now calculate,

\begin{align*}
A^{p}&=A^{p}\|x\|_{L_2}^p\\
&\leq \|x\|_{L_p}^p\\
&= \left(\int_{|x|\geq a} |x|^p\right)+\left(\int_{|x|< a} |x|^p\right)\\
&\leq \left(\int 1_{|x|\geq a} |x|^p\right)+ a^p\\
&\leq \left(\int 1_{|x|\geq a}\right)^{1/r'}\left(\int |x|^{pr}\right)^{1/r}+a^p\qquad\big(\textrm{by Hölder's inequality}\big)\\
&= \textrm{Prob}(|x|\geq a)^{1/r'}\left(\int |x|^{2}\right)^{p/2}+  a^p\\
&= \textrm{Prob}(|x|\geq a)^{1-p/2}\|x\|^p_{L_2}+ a^p\\
&= \textrm{Prob}(|x|\geq a)^{1-p/2}+a^p
\end{align*}

As the exponent $1-p/2$ is positive,  solving for $\textrm{Prob}(|x|\geq a)$ gives,
$$\textrm{Prob}(|x|\geq a)\geq (A^{p}-a^p)^{1/(1-p/2)}=(A^{p}-a^p)^{2/(2-p)}.
$$

Thus, we have proven the lemma  for $\gamma=(A^{p}-a^p)^{2/(2-p)}$.

\end{proof}

We say that a  sequence $(y_j)_{j=1}^\infty$ in a real Banach space $X$ is $C$-{\em unconditional} if for any finite sequence of scalars $(a_j)_{j=1}^N$ and any choice of signs $\vp_j=\pm1$ we have that 
$$\Big\|\sum_{j=1}^N  \vp_j a_j y_j\Big\|_X\leq C \Big\|\sum_{j=1}^N a_j y_j\Big\|_X.
$$
We say $(y_j)_{j=1}^\infty$ is $C$-{\em suppression unconditional} if for any finite sequence of scalars $(a_j)_{j=1}^N$ and any $J\subseteq\{1,...,N\}$ we have that 
$$\Big\|\sum_{j\in J} a_j y_j\Big\|_X\leq C \Big\|\sum_{j=1}^N a_j y_j\Big\|_X.
$$
It follows from the triangle inequality that if $(y_j)_{j=1}^\infty$ is $C$-suppression unconditional then it is $2C$-unconditional.  Indeed,
$$\Big\|\sum_{i=1}^N \vp_j a_j y_j\Big\|\leq \Big\|\sum_{\vp_j=1} a_j y_j\Big\|+\Big\|\sum_{\vp_j=-1} a_j y_j\Big\|\leq 2C\Big\|\sum_{j=1}^N a_j y_j\Big\|.
$$

\begin{lem}\label{L:unc}
Let $(y_j)_{j=1}^\infty$ be a sequence of independent mean zero random variables in $L_2([0,1])$.  Then $(y_j)_{j=1}^\infty$ is 1-suppression unconditional as a sequence in $L_1([0,1])$.  Hence, $(y_j)_{j=1}^\infty$ is 2-unconditional as a sequence in $L_1([0,1])$. 
\end{lem}

\begin{proof}
Let $(a_j)_{j=1}^N$ be a finite sequence of scalars and let $J\subseteq\{1,...,N\}$.  For all $t\in[0,1]$ we let $f_J(t)=sign(\sum_{j\in J} a_j y_j(t))$.  Note that $f_J$ is a random variable which is independent to $\sum_{j\not\in J } a_j y_j$.  This implies that $\int_0^1 f_J(t)\sum_{j\not\in J} a_j y_j(t) dt=0$ as each $y_j$ is mean zero.
\begin{align*}
\big\|\sum_{j\in J} a_j y_j\big\|_{L_1}&= \int_0^1 \big|\sum_{j\in J} a_j y_j(t)\big|\, dt\\
&= \int_0^1 f_J(t)\sum_{j\in J} a_j y_j(t)\, dt\\
&= \int_0^1 f_J(t)\sum_{j=1}^N a_j y_j(t)\, dt\quad\big(\textrm{as $\int_0^1 f_J(t)\sum_{j\not\in J} a_j y_j(t) dt=0$}\big)\\
&\leq \int_0^1 \Big|\sum_{j=1}^N a_j y_j(t)\Big|\, dt= \Big\|\sum_{j=1}^N a_j y_j\Big\|_{L_1}
\end{align*}
Thus, the sequence $(y_j)_{j=1}^\infty$ is 1-suppression unconditional in $L_1([0,1])$.

\end{proof}

The {\em Rademacher sequence} $(r_j)_{j=1}^\infty$  is an independent sequence of mean-zero $\pm1$ random variables on $[0,1]$.  This sequence can be very useful when studying unconditionality in Banach spaces, and we use the Rademacher sequence in the following case of Khintchine's inequality. 

\begin{thm}[Khintchine's Inequality]
There exists a constant $A_1>0$ such that for all $N\in\N$ and all scalars $(a_j)_{j=1}^N$, we have that
$\int_0^1 |\sum_{j=1}^N a_j r_j(t)| dt\geq A_{1} (\sum_{j=1}^N |a_j|^2)^{1/2}.
$
\end{thm}

The following lemma is a useful technique which is widely used in the study of Banach lattices.

\begin{lem}\label{L:ind_lower}
Let $(y_j)_{j=1}^\infty$ be an orthonormal sequence of independent mean zero random variables in $L_2([0,1])$ such that there exists $c>0$ with $\|y_j\|_{L_1}\geq c$ for all $j\in \N$.  Then, $\|x\|_{L_1}\geq \frac{1}{2}cA_1 \|x\|_{L_2}$ for all $x\in \overline{span}(y_j)$, where $A_1$ is the constant given in Khintchine's inequality.
\end{lem}
\begin{proof}
Let $x=\sum_{j=1}^N a_j y_j\in span(y_j)$.  We have that,
\begin{align*}
\|x\|_{L_1}&=\int_0^1 \Big|\sum_{j=1}^N a_j y_j(s)\Big|\,ds\\
&\geq \frac{1}{2}  2^{-N}\sum_{\vp_j=\pm1} \int_0^1 \Big|\sum_{j=1}^N \vp_j a_j y_j(s)\Big|\,ds\quad\big(\textrm{as $(y_j)$ is 2-uncondtitional in $L_1$}\big)\\
&=\frac{1}{2}  \int_0^1 \int_0^1 \Big|\sum_{j=1}^N r_j(t)a_j y_j(s)\Big|\,ds\,dt \quad\big(\textrm{where $(r_j)$ is the Rademacher sequence}\big)\\
&=\frac{1}{2}  \int_0^1 \int_0^1 \Big|\sum_{j=1}^N r_j(t)a_j y_j(s)\Big|\,dt\,ds\\
&\geq\frac{1}{2}A_1  \int_0^1  \Big(\sum_{j=1}^N |a_j y_j(s)|^2\Big)^{1/2}\,ds\qquad\big(\textrm{by Khintchine's Inequality}\big)\\
&\geq\frac{1}{2}A_1   \Big(\sum_{j=1}^N\Big( \int_0^1 |a_j y_j(s)|\,ds\Big)^2\Big)^{1/2}\qquad\big(\textrm{by Jensen's Inequality}\big)\\
&=\frac{1}{2}A_1   \Big(\sum_{j=1}^N|a_j|^2  \|y_j\|_{L_1}^2\Big)^{1/2}\\
&\geq\frac{1}{2}A_1 c   \Big(\sum_{j=1}^N|a_j|^2  \Big)^{1/2}=\frac{1}{2}A_1 c   \|x\|_{L_2}
\end{align*}
\end{proof}

We are now prepared to prove the remaining parts of Theorem \ref{T:main}.

\begin{thm}\label{T:main3}
Let $(y_j)_{j=1}^\infty$ be an independent orthonormal sequence of mean zero random variables in $L_2([0,1])$ and  let $x_j=y_j+1_{(j,j+1]}$ for all $j\in\N$.  Then the following are equivalent,
\begin{enumerate}
\item The subspace $\overline{span} (x_j)_{j=1}^\infty$ does stable phase retrieval in $L_2(\R)$.
\item There exist constants $a,\gamma>0$ with $\textrm{Prob}(|x|\geq a\|x\|)\geq\gamma$ for all $x\in \overline{span}(y_j)_{j=1}^\infty$.
\item  The $L_1$ and $L_2$ norms of $y_j$ are comparable.  That is, there exists $A>0$ such that  $\|y_j\|_{L_2([0,1]}\geq \|y_j\|_{L_1([0,1])}\geq A \|y_j\|_{L_2([0,1])}$ for all $j\in\N$.
\item It is not the case that for all $\vp>0$, the subspace $\overline{span}(y_j)_{j=1}^\infty\subseteq L_2([0,1])$ contains two independent unit vectors which may be perturbed by less than $\vp$ to have disjoint support.  
\end{enumerate}
\end{thm}

\begin{proof}
We have by Theorem \ref{T:main2} that $(2)\Rightarrow(1)$, and we have by Lemma \ref{L:ind_lower} and Lemma \ref{L:bb2} that $(3)\Rightarrow(2)$.

We now prove that $(1)\Rightarrow (4)$ by contrapositive.   Let $\vp>0$. We assume that there exists $J\subseteq\N$ and unit vectors $x\in {span}(y_j)_{j\in J}$ and $y\in {span}(y_j)_{j\in\N\setminus J}$ such that there exists $x',y'\in L_2([0,1])$ with disjoint support satisfying $\|x-x'\|<\vp$ and $\|y-y'\|<\vp$.

We have that $x=\sum_{j\in J} a_j y_j$ and  $y=\sum_{j\in J^c} a_j y_j$ for some sequence of scalars $(a_j)_{j=1}^\infty$.  Let  $f=\sum_{j\in J}a_j (y_j+1_{(j,j+1]})$ and $g=\sum_{j\in J^c}a_j (y_j+1_{(j,j+1]})$.  As $\|x\|=\|y\|=1$ we have that $\|f\|=\|g\|=\sqrt{2}$. 
Thus, $\|(f+g)-(f-g)\|=2\|g\|=2\sqrt{2}$ and 
$\|(f+g)+(f-g)\|=2\|f\|=2\sqrt{2}$.

On the other hand, we have that
\begin{align*}
\big\||f+g|&-|f-g|\big\|=\Big\| \big|\sum a_j (y_j+1_{(j,j+1]})\big|-\big|\sum_{j\in J} a_j (y_j+1_{(j,j+1]})-\sum_{j\not\in J} a_j (y_j+1_{(j,j+1]})\big|\Big\|\\
&=\Big\| \big|\sum a_j y_j\big|-\big|\sum_{j\in J} a_j y_j-\sum_{j\not\in J} a_j y_j\big|\Big\|\\
&=\big\| |x+y|- |x-y|\big\|\\
&= \big\| |x-x'+x'+y'-y'+y|-|x-x'+x'-y'+y'-y|\big\|\\
&\leq \big\| |x'+y'|-|x'-y'|\big\|+2\|x-x'\|+2\|y-y'\|\\
&= \big\| |x'+y'|-|x'+y'|\big\|+2\|x-x'\|+2\|y-y'\|\quad\textrm{ (as $x'$ and $y'$ has disjoint support)}\\
&<0+2\vp+2\vp
\end{align*}
Thus, $\big\||f+g|-|f-g|\big\|<4\vp$.  This proves that phase retrieval for $\overline{span} (y_j+1_{(j,j+1]})_{j=1}^\infty$ is not stable in $L_2(\R)$ and contradicts (1).

We now prove $(4)\Rightarrow (3)$ by contrapositive.  We assume that for all $A>0$ there exists $j\in\N$ with $\|y_j\|_{L_1([0,1])}\leq A$.  Let $\vp>0$.  
Choose $N\in\N$ such that $\|y_N\|_{L_1([0,1])}<\vp^2$.  We have that $\|y_N\|_{L_1([0,1])}\geq \vp \textrm{Prob}(y_N\geq\vp)$.  Thus, $\textrm{Prob}(y_N\geq\vp)<\vp$.  Let $x'=y_1 \cdot 1_{y_N<\vp}$ and $y'=y_N \cdot 1_{y_N\geq \vp}$.  Thus $x'$ and $y'$ have disjoint support.  
We have that
$$
\|y_N-y'\|_{L_2([0,1])}=\|y_N\cdot 1_{y_N<\vp}\|_{L_2([0,1])}<\vp.
$$
For estimating the distance between $y_1$ and $x'$ we have that
\begin{align*}
\|y_1-x'\|^2_{L_2([0,1])}&=\|y_1\cdot 1_{y_N\geq \vp}\|^2_{L_2([0,1])}\\
&=\int_{y_N\geq \vp}|y_1|^2 \\
&=\textrm{Prob}(y_N\geq \vp)\int|y_1|^2\qquad\textrm{ (as $y_1$ and $y_N$ are independent)} \\
&<\vp
\end{align*}

Hence $x'$ is within $\vp^{1/2}$ of $y_1$, $y'$ is within $\vp$ of $y_N$, and the vectors $x'$ and $y'$ have disjoint support.  As $\vp>0$ is arbitrary, this contradicts (4).
\end{proof}

\section{Stability in finite dimensions for sub-Gaussian random variables}\label{S:4}

 As frames and continuous frames never do stable phase retrieval for infinite dimensional Hilbert spaces, we had to express our results in terms of stable phase retrieval for subspaces of $L_2$.  However, when proving stable phase retrieval for finite dimensional spaces, we are able to express our results in terms of frames and continuous frames.
When constructing frames which do phase retrieval for finite dimensional Hilbert spaces, one desires a method which produces a frame with low stability bound which does not use a large number of vectors relative to the dimension.
In particular, we are interested in constructions for which there exists a constant $c>0$ so that for every dimension $n\in\N$ the construction will produce a frame $(f_j)_{j=1}^m$ for $\ell_2^n$ which does $c$-stable phase retrieval where $m$ is on the order of $n$. 
As is the case for RIP matrices from compressed sensing \cite{BDDW},  all known constructions which achieve optimal $m\in\N$ relative to $n\in\N$ are random and achieve their result with high probability.

  The first result on uniform stability for phase retrieval in an $n$-dimensional Hilbert space was for $m$ iid random vectors with uniform distribution on the sphere where $m$ was  on the order of $n\log(n)$ \cite{CSV}, which was later improved to $m$ being on the order of $n$ \cite{CL}.  This achievement was then extended to other sub-Gaussian distributions with additional assumptions such as small-ball probability assumptions \cite{EM} or lower $L_\infty$ bounds \cite{KL}\cite{KS}.  In the previous sections we gave a construction which did stable phase retrieval for an infinite dimensional subspace of $L_2(\R)$.  Note that this gives uniformly stable phase retrieval for each finite dimensional subspace.  By considering each finite dimensional subspaces as the range of the analysis operator of a continuous frame, we obtain  continuous frames for $\ell_2^n$ which do uniformly stable phase retrieval.  We then prove that if the continuous frame is constructed using sub-Gaussian random variables then a sampling of $m$ on the order of $n$ random vectors forms a frame of $\ell_2^n$ which does uniformly stable phase retrieval with high probability.  Recall that the corresponding results in infinite dimensions are impossible as no frame or continuous frame does stable phase retrieval for an infinite dimensional Hilbert space \cite{CCD}\cite{AG}. 
  Given a measure space $\Omega$ and a Hilbert space $H$, we say that $(f_t)_{t\in\Omega}\subseteq H$ is a {\em continuous Parseval frame} of $H$ if $\int_\Omega |\langle x, f_t\rangle|^2 dt=\|x\|^2$ for all $x\in H$.
  
  \begin{thm}\label{T:main_P_cframe}
Let $(y_j)_{j=1}^n$ be an orthonormal sequence of independent, mean-zero random variables in $L_2([0,1])$.  Let $a,\gamma>0$ so that $\textrm{Prob}(|x|\geq a\|x\|)_{L_2})\geq\gamma$ for all $x\in span(y_j)_{j=1}^n.$  Let $v_t=(y_1(t),...,y_n(t))\in\ell_2^n$ for all $t\in[0,1]$ and let $(e_j)_{j=1}^n$ be the unit vector basis for $\ell_2^n$.  Then $(\frac{1}{\sqrt{2}}v_t)_{t\in[0,1]}\cup (\frac{1}{\sqrt{2}}e_j)_{j=1}^n$ is a continuous Parseval frame of $\ell_2^n$ which does $6a^{-1}\gamma^{-1}$-stable phase retrieval.   That is, for all $x,y\in \ell_2^n$ we have that
$$min(\| x-y\|,\| x+y\|) \leq 6a^{-1}\gamma^{-1}
\Big(
\sum_{j=1}^n \big(|\langle x,e_j\rangle|-|\langle y, e_j\rangle|\big)^2+
\int_0^1\big(|\langle x,v_t\rangle|-|\langle y, v_t\rangle|\big)^2dt\Big)^{1/2}.
$$
\end{thm}

\begin{proof}
Let $x=(a_j)_{j=1}^n\in\ell_2^n$.  As $(y_j)_{j=1}^n$ are ortho-normal we have that
$$\int_{0}^1 |\langle x,v_t\rangle|^2\,dt=\int_0^1 \big|\sum_{j=1}^n a_j y_j(t)\big|^2\,dt
=\Big\| \sum_{j=1}^n a_j y_j(t)\Big\|^2_{L_2([0,1])}=\sum_{j=1}^n |a_j|^2=\|x\|^2.
$$
Thus, $(v_t)_{t\in[0,1]}$ is a continuous Parseval frame of $\ell_2^n$.  Hence, $(\frac{1}{\sqrt{2}}v_t)_{t\in[0,1]}\cup(\frac{1}{\sqrt{2}}e_{\lfloor t\rfloor})_{t\in(1,n+1)}$ is a continuous Parseval frame of $\ell_2^n$.  If $T:\ell_2^n\rightarrow L_2([0,n+1))$ is the analysis operator of $(\frac{1}{\sqrt{2}}v_t)_{t\in[0,1]}\cup(\frac{1}{\sqrt{2}}e_{\lfloor t\rfloor})_{t\in(1,n+1)}$ then $T(e_j)=\frac{1}{\sqrt{2}}y_j+\frac{1}{\sqrt{2}}1_{[j,j+1)}$ for all $1\leq j\leq n$.  By Theorem \ref{T:main} we have that the subspace $span(y_j+1_{[j,j+1)})_{j=1}^n\subseteq L_2([0,n+1))$ does 
$ 6a^{-1}\gamma^{-1} 
$-stable phase retrieval.  Thus the continous Parseval frame $(\frac{1}{\sqrt{2}}v_t)_{t\in[0,1]}\cup(\frac{1}{\sqrt{2}}e_{\lfloor t\rfloor})_{t\in(1,n+1)}$ does $ 6a^{-1}\gamma^{-1} 
$-stable phase retrieval.  
\end{proof}

Now that we have continuous frames for $\ell_2^n$ which do uniformly stable phase retrieval, we can consider the problem of sampling the continuous frames to obtain  frames for $\ell_2^n$ which do uniformly stable phase retrieval.  In the solution to the Discretization Problem, Speegle and the third author characterize when a continuous frame for a Hilbert space $H$ may be sampled to obtain a frame for $H$ \cite{FS}.  However, the proof relies on the solution to the Kaddison-Singer problem by Marcus, Spielman, and Srivistava \cite{MSS}, and though it is possible to sample a continuous frame to obtain a discrete frame with good frame bounds,  this may occur only with low probability.  However, if a continuous Parseval frame for $\ell_2^n$ is sub-Gaussian, then with high probability one can use $m$ on the order of $n$ sampling points to obtain a frame which with high probability has upper frame bound $2$ and lower frame bound $1/2$.  We will apply this fact to prove in Theorem \ref{T:main_P_sG} that if our construction uses sub-Gaussian random variables then the continuous frame in Theorem \ref{T:main_P_cframe}  may be sampled using $m$ on the order of $n$ points to provide a frame which with high probability does stable phase retrieval with stability constant independent of the dimension.

We introduce notation from \cite{V} and will use theorems from there on sub-Gaussian random variables and concentration inequalities.  We use the term {\em random variable} to refer to a measurable function from a probability space to $\R$, and we use the term {\em random vector} to refer to a measurable function from a probability space to a vector space.  There are many equivalent definitions for a random variable to be sub-Gaussian, and the following will be most convenient for us.
\begin{defn}
We say that a random variable $X$ is {\em $K$-sub-Gaussian} for some constant $K>0$ if 
$$\|X\|_{L_p}\leq K \sqrt{p}\hspace{1cm}\textrm{ for all }p\geq 1.
$$
 If $V$ is a random vector in a Hilbert space $H$, we say that $V$ is {\em $K$-sub-Gaussian} if the random variable $\langle x, V\rangle$ is $K$-sub-Gaussian for all $x\in H$ with $\|x\|=1$.  We say that a random vector $V$ in a Hilbert space $H$ is {\em isotropic} if $\E |\langle x, V\rangle|^2=\|x\|^2$ for all $x\in H$.  
\end{defn} 

We will be considering the case that $X$ is a $K$-sub-Gaussian random variable such that $\|X\|_{L_2}=1$ and hence we will always have that $K\geq 2^{-1/2}$.
Note that $V$ is an isotropic random vector is equivalent to $V$ being a continuous Parseval frame  over a probability space.
Most of the theorems about sub-Gaussian random variables and sub-Gaussian random vectors that we cite include a universal constant $C>1$.  We will use $C$ to denote this constant for the remainder of this section.  The following lemma gives a simple criterion to test if a random vector is sub-Gaussian.

\begin{lem}[Theorem 3.4.2 \cite{V}] \label{T:Vv}
Let $V=(y_1,y_2,...,y_n)$ be a random vector in $\ell_2^n$.  If the coordinates $(y_j)_{j=1}^n$ are independent, mean-zero, variance-one, and $K$-sub-Gaussian random variables then $V$ is a $CK$-sub-Gaussian isotropic random vector.
\end{lem}

Let $(y_j)_{j=1}^\infty$ be an ortho-normal sequence of independent mean zero random variable in $L_2([0,1])$.  
Our characterization of phase retrieval in the previous sections showed that the subspace $\overline{span}(y_j+1_{(j,j+1)})$ does stable phase retrieval in $L_2(\R)$ if and only if there exists constants $a,\gamma>0$ such that $\textrm{Prob}(|x|\geq a\|x\|)\geq\gamma$ for all $x\in\overline{span}(y_j)_{j=1}^\infty$.  The following lemma shows that if $(y_j)_{j=1}^\infty$ are uniformly sub-Gaussian then they satisfy this property.  Recall that we use $A_1$ to denote the constant used in Khintchine's inequality. 

\begin{lem}\label{L:KL1}
If $(y_j)_{j=1}^\infty$ is a sequence independent, mean-zero, variance-one, and $K$-sub-Gaussian random variables then for $a=16^{-1}A_1 K^{-2}$ and $\gamma=a^2$ we have that 
$$\textrm{Prob}(|x|\geq a\|x\|_{L_2})\geq\gamma \qquad\textrm{ for all }x\in span(y_j).
$$
Thus for all $n\in\N$, the random vector $v=(y_1,...,y_n)$ in $\ell_2^n$ satisfies
$$\textrm{Prob}(|\langle x, v\rangle|\geq a\|x\|_{\ell_2^n})\geq\gamma\qquad\textrm{ for all }x\in \ell_2^n.$$
\end{lem}

\begin{proof}
We have for  $j\in\N$ and $2< p<\infty$ that $\|y_j\|_{L_p}\leq K\sqrt{p}$ and $\|y_j\|_{L_2}=1$. 
We will obtain our result by considering specifically $p=4$, but we will wait till the end before making that substitution.
Fix $\lambda=\frac{p-2}{p-1}$ and note that $0<\lambda<1$ and that $1=\frac{2-\lambda}{p}+\frac{\lambda}{1}$. 
We now have for each $j\in\N$ that
\begin{align*}
    \|y_j\|^2_{L_2}&=\int |y_j|^{\lambda}|y_j|^{2-\lambda}\\
    &\leq\Big(\int |y_j|\Big)^{\lambda}\Big(\int |y_j|^{p}\Big)^{(2-\lambda)/p}\qquad \big(\textrm{by H{\"o}lder's inequality for $\tfrac{1}{\lambda}$ and $\tfrac{p}{2-\lambda}$}\big)\\
    &= \|y_j\|_{L_1}^\lambda \|y_j\|_{L_p}^{2-\lambda}
\end{align*}

Thus, as $\|y_j\|_{L_2}=1$ and $\|y_j\|_{L_p}\leq K p^{1/2}$ we have that 
$$\|y_j\|_{L_1}\geq K^{1-2/\lambda}p^{1/2-1/\lambda}
$$
By Lemma \ref{L:ind_lower} we have that 

$$\|x\|_{L_1}\geq \frac{1}{2}A_1 K^{1-2/\lambda}p^{1/2-1/\lambda} \|x\|_{L_2}\qquad\textrm{ for all $x\in span(y_j)$.}$$
We now apply Lemma \ref{L:bb2} for $a=\frac{1}{4}A_1 K^{1-2/\lambda}p^{1/2-1/\lambda}$ and $\gamma=a^2$ to get that 
$$\textrm{Prob}(|x|\geq a\|x\|_{L_2})\geq\gamma \qquad\textrm{ for all }x\in span(y_j).
$$
Plugging in $p=4$ gives $\lambda=2/3$ and hence $a=16^{-1}A_1 K^{-2}$.
\end{proof}

A collection of vectors $(f_j)_{j\in J}$ in a Hilbert space $H$ is called a {\em frame} of $H$ with {\em frame bounds $A$ and $B$} if $0<A\leq B<\infty$ and 
$$A\|x\|^2\leq \sum_{j\in J} |\langle x, f_j\rangle|^2\leq B\|x\|^2\qquad\textrm{ for all }x\in H.
$$
The {\em analysis operator} of a frame $(f_j)_{j\in J}$ of $H$ is the map $T:H\rightarrow \ell_2(J)$ given by $T(x)=(\langle x, f_j\rangle)_{j\in J}$ for all $x\in H$.  We will use in particular that if $(f_j)_{j\in J}$ has upper frame bound $B$ then the operator norm of the analysis operator satisfies $\|T\|\leq B^{1/2}$.

Suppose that $v$ is an isotropic random vector in $\ell^n_2$ and that $(v_j)_{j=1}^m$ are independent copies of $v$. Then typically, one must choose $m$ on the order of $n\log (n)$ so that with high probability the frame $(m^{-1/2} v_j)_{j=1}^m$ has lower frame bound $\frac{1}{2}$ and upper frame bound $2$ \cite{R}.  However, the following theorem will allow us to prove that if $(v_j)_{j=1}^m$ are sub-Gaussian then we may choose $m$ on the order of $n$ so that with high probability the frame $(m^{-1/2} v_j)_{j=1}^m$ has lower frame bound $\frac{1}{2}$ and upper frame bound $2$.
 Recall that the constant $C$ used in the following theorem is the same universal constant that we used earlier this section.

\begin{thm}[Theorem 4.6.1 \cite{V}] \label{T:sG}
Let $(f_j)_{j=1}^m$ be a sequence of independent, mean-zero, $K$-sub-Gaussian, isotropic random vectors in $\ell_2^n$.  Then for any $s> 0$ we  have with probability at least $1-2\exp(-s^2)$ that
$$\big(\sqrt{m}-CK^2(\sqrt{n}+s)\big)^2  \|x\|^2\leq \sum_{j=1}^m |\langle x,f_j\rangle|^2\leq  \big(\sqrt{m}+CK^2(\sqrt{n}+s)\big)^2\|x\|^2
$$
\end{thm}

By applying Theorem \ref{T:sG}  we get the following corollary.

\begin{cor}\label{C:frame_sG}
Let $(y_j)_{j=1}^n$ be an independent sequence of  mean-zero, variance-one, $K$-sub-Gaussian random variables.  Let $(v_j)_{j=1}^m$ be an independent sequence in $\ell_2^n$ with the same distribution as $v=(y_1,y_2,...,y_n)$. 
   If $m\geq 64C^6K^4n$ then with probability at least $1-2\exp(-64^{-1}C^{-6}K^{-4}m)$ we have that $(m^{-1/2}v_j)_{j=1}^m$ is a frame of $\ell_2^n$ with lower frame bound $1/2 $ and upper frame bound $2$.  
\end{cor}

\begin{proof}
By Lemma \ref{T:Vv} we have that $(v_j)_{j=1}^m$ is a sequence of independent, mean-zero, isotropic, $KC$-sub-Gaussian random vectors.  

We apply Theorem \ref{T:sG} except with sub-Gaussian constant $CK$ instead of $K$.  Thus, for all $s>0$ we have with probability at least $1-2\exp(-s^2)$  that $(v_j)_{j=1}^m$ has lower frame bound $A$ and upper frame bound $B$, where  
$$A:= \big(\sqrt{m}-C^3K^2(\sqrt{n}+s)\big)^2\qquad\textrm{ and }\qquad
B:=\big(\sqrt{m}+C^3K^2(\sqrt{n}+s)\big)^2.$$  
       We now do some substitutions to get our desired formulation.
Suppose that $m\geq 64C^6K^4n$.  We let $s=4^{-1}C^{-3}K^{-2}m^{1/2}-n^{1/2}$.  Note that $s>0$.
We have that 
$$s^2=(4^{-1}C^{-3}K^{-2}m^{1/2}-n^{1/2})^2\geq (4^{-1}C^{-3}K^{-2}m^{1/2}-8^{-1}C^{-3}K^{-2}m^{1/2})^2
=64^{-1}C^{-6}K^{-4}m
$$
Thus, with probability at least $1-2\exp(-64^{-1}C^{-6}K^{-4}m)$ we have that $(v_j)_{j=1}^m$ has lower frame bound $A$ and upper frame bound $B$.
We now show that $B< 2m$.
\begin{align*}
B^{1/2}&= m^{1/2}+C^3K^2(n^{1/2}+s)\\
&=m^{1/2} + C^3K^2 (4^{-1}C^{-3}K^{-2}m^{1/2})\\
&=(5/4)m^{1/2}
\end{align*}
Thus, $B= (25/16)m<2m$.  The same argument gives that $A=(9/16)m> m/2$.  By scaling, we have that $(m^{-1/2}v_j)_{j=1}^m$ has lower frame bound $1/2$ and upper frame bound $2$.
\end{proof}

We now set some notation and give motivation for how we will proceed.  Let $(y_j)_{j=1}^n$ be an independent sequence of  mean-zero, variance-one, $K$-sub-Gaussian random variables and let $a,\gamma>0$ be constants such that
$\textrm{Prob}(|y|\geq a\|y\|_{L_2})\geq\gamma$ for all $y\in span(y_j)_{1\leq j\leq n}$.    Let $(v_j)_{j=1}^m$ be an independent sequence in $\ell_2^n$ with the same distribution as $v=(y_1,y_2,...,y_n)$.  Thus, for all $x\in \ell_2^n$ we have that  
\begin{equation}\label{E:low}
\textrm{Prob}\big(|\langle x,v\rangle|\geq a\|x\|\big)\geq\gamma \quad\textrm{ for all }x\in\ell_2^n.
\end{equation}
For each $x,y\in\ell_2^n$ with disjoint support and each $b>0$, we denote $J_{x,y}(b)$ to be the set
\begin{equation}\label{E:J_sG}
J_{x,y}(b)=\Big\{j\in[m]\,:\,|\langle x,v_j\rangle|\geq b\|x\|\textrm{ and }|\langle y,v_j\rangle|\geq b\|y\|\Big\},
\end{equation}
  where we denote $[m]=\{1,2,...,m\}$. 
By \eqref{E:low} we have  $\E|J_{x,y}(a)|\geq\gamma^2 m$. 
We will prove that with high probability, every $x,y\in \ell_2^n$ with disjoint support satisfies $|J_{x,y}(a/2)|\geq\gamma^2 m/4$.  Our first step is proving
the following Lemma which gives that for each individual $x,y\in\ell_2^n$ with disjoint support we have that $|J_{x,y}(a)|\geq\gamma^2 m/2$ with high probability.

\begin{lem}\label{L:lower_bP_sG}
Let $(y_j)_{j=1}^n$ be an independent sequence of  
random variables.  Let $(v_j)_{j=1}^m$ be an independent sequence in $\ell_2^n$ with the same distribution as $v=(y_1,y_2,...,y_n)$.  Suppose that $a,\gamma>0$ are such that 
$\textrm{Prob}(|\langle x, v\rangle|\geq a\|x\|_{L_2})\geq\gamma$ for all $x\in \ell_2^n$. Then
$$\textrm{Prob}\left(|J_{x,y}(a)|\leq (\gamma^2/2)m\right)\leq e^{-\gamma^4 m/2}\qquad\textrm{ for all $x,y\in \ell_2^n$ with disjoint support}.
$$

\end{lem}

\begin{proof}
As $x$ and $y$ are supported on disjoint coordinates of $\ell_2^n$ and the coordinates of $v$ are independent, we have that
$$\textrm{Prob}\Big(|\langle x,v\rangle |\geq a\|x\|\textrm{ and }|\langle y,v\rangle|\geq a \|y\|\Big)=\textrm{Prob}\Big(|\langle x,v\rangle |\geq\|x\|\Big)\textrm{Prob}\Big(|\langle y,v\rangle |\geq a\|y\|\Big)\geq \gamma^2
$$
Thus, for each $1\leq j\leq m$ we have that
$$
\textrm{Prob}\Big(j\in J_{x,y}(a)\Big)= \textrm{Prob}\Big(|\langle x,v_j\rangle|\geq a\|x\|\textrm{ and }|\langle y,v_j\rangle|\geq a \|y\|\Big)\geq \gamma^2.$$
We now consider a sequence $(b_j)_{j=1}^m$ of iid Bernoulli random variables by setting
 $b_j=1$ if $j\in J_{x,y}(a)$ and $b_j=0$ if $j\not\in J_{x,y}(a)$.  Each $b_j$ has expectation at least $\gamma^2$.    By Hoeffding's inequality we have for all $\lambda>0$ that
$$\textrm{Prob}\Big(|J_{x,y}(a)|\leq (\gamma^2-\lambda)m\Big)= \textrm{Prob}\left(\sum_{j=1}^m b_j\leq (\gamma^2-\lambda)m\right) \leq e^{-2\lambda^2 m}.
$$
We choose $\lambda=\gamma^2/2$ to obtain our desired result.
\end{proof}

We now set more notation.  Let $(e_j)_{j=1}^n$ be the unit vector basis for $\ell_2^n$.
Let $1>\vp>0$.  For all $k\in\N$, the unit sphere of $\ell_2^k$ may be covered by $(3/\vp)^k$ balls of radius $\vp$ (Corollary 4.2.13 in \cite{V}). Thus, for all $I\subseteq [n]$ there exists a set of unit vectors $Z_{I,\vp}\subseteq span(e_j)_{j\in I}$ such that $Z_{I,\vp}$ is $\vp$-dense in the unit sphere of $span(e_j)_{j\in I}$ and $|Z_{I,\vp}|\leq (3/\vp)^{|I|}$.
We let $D_\vp:=\cup_{I\subseteq[n]}Z_{I,\vp}\times Z_{I^c,\vp}$.   Note that Lemma \ref{L:lower_bP_sG} applies to each fixed but arbitrary pair $(x,y)$ with $x\in span(e_j)_{j\in I}$ and $y\in span(e_j)_{j\in I^c}$ for some $I\subseteq[n]$. 
The following corollary extends the result to all  $(x,y)\in D_\vp$.

\begin{cor}\label{C:lower_bp_sG}
Let $(y_j)_{j=1}^n$ be an independent sequence of   random variables.  Let $(v_j)_{j=1}^m$ be an independent sequence in $\ell_2^n$ with the same distribution as $v=(y_1,y_2,...,y_n)$.  Suppose that $a,\gamma>0$  are such that 
$\textrm{Prob}(|\langle x, v\rangle|\geq a\|x\|_{L_2})\geq\gamma$ for all $x\in \ell_2^n$. Then with probability at least $1-\exp\big((\log(6\vp^{-1})n-2^{-1}\gamma^4 m \big)$, we have that
$$|J_{x,y}(a)|\geq (\gamma^2/2)m\quad\textrm{ for all }(x,y)\in D_\vp.$$
\end{cor}

\begin{proof}

  The cardinality of $D_\vp$ is at most 
$$
|D_\vp|\leq \sum_{I\subseteq[n]}  |Z_{I,\vp}| |Z_{I^c,\vp}| \leq \sum_{I\subseteq[n]}  \left(\frac{3}{\vp}\right)^{|I|}\left(\frac{3}{\vp}\right)^{|I^c|}=2^n 3^{n} \vp^{-n}=6^n \vp^{-n}.
$$

By using Lemma \ref{L:lower_bP_sG} and a union bound, we have that the probability that 
$|J_{x,y}(a)|\geq (\gamma^2/2)m$ for every $(x,y)\in D_\vp$ is at least,
\begin{align*}
1-\sum_{(x,y)\in D_\vp}\textrm{Prob}\Big(|J_{x,y}(a)|\leq (\gamma^2/2)m\Big)&\geq 1-|D_\vp| e^{-2^{-1} \gamma^4 m}\\
&\geq 1-6^n\vp^{-n} e^{-2^{-1}\gamma^4 m}\\
&=1-\exp\left((\log(6\vp^{-1})n-2^{-1}\gamma^4 m \right).
\end{align*}
\end{proof}

The following theorem extends Corollary \ref{C:lower_bp_sG} to all of $\ell_2^n$.

\begin{thm}\label{T:all}
Let $m,n\in\N$ and $\vp,a,\gamma>0$ such that $\vp\leq 8^{-1} \gamma a$. Suppose that $(m^{-1/2} v_j)_{j=1}^m$ is a frame of $\ell_2^n$ with upper frame bound $2$ and that $|J_{x,y}(a)|\geq \gamma^2 m /2$ for all $(x,y)\in D_\vp.$ Then, $|J_{x,y}(a/2)|\geq \gamma^2 m/4$ for all $x,y\in \ell_2^n$ with disjoint support.

\end{thm}
\begin{proof}
Let $x,y\in \ell_2^n$ with $x\in span_{j\in I}e_j$ and $y\in span_{j\in I^c}e_j$ for some $I\subseteq[n]$.  By scaling, we may assume without loss of generality that $\|x\|=\|y\|=1$.  Choose some $(x_0,y_0)\in D_\vp$ with $\|x-x_0\|<\vp$ and $\|y-y_0\|<\vp$.  For the sake of contradiction, we assume that $|J_{x,y}(a/2)|<\gamma^2 m/4$.  Thus, $|J_{x_0,y_0}(a)\setminus J_{x,y}(a/2)|>\gamma^2 m/4$.  We have that either $|\{j\in [m]:|\langle x_0-x,v_j\rangle|>a/2\}|> \gamma^2 m/8$ or $|\{j\in [m]:|\langle y_0-y,v_j\rangle|>a/2\}|> \gamma^2 m/8$.  We let $J=\{j\in [m]:|\langle x_0-x,v_j\rangle|>a/2\}$ and assume without loss of generality that $|J|>\gamma^2 m/8$.

Let $T:\ell_2^n\rightarrow \ell_2^m$ be the the analysis operator of the frame $(m^{-1/2}v_{j})_{j=1}^m$.  That is, $T(z)=\big(\langle z, m^{-1/2}{v_j}\rangle\big)_{j=1}^m$ for all $z\in \ell_2^n$.  As $(m^{-1/2}v_j)_{j=1}^m$ has upper frame bound $2$, we have that $T$ has operator norm $\|T\|\leq 2^{1/2}$.  We now have that
$$
 \|T(x-x_0)\|^2\leq 2\|x-x_0\|^2< 2\vp^2.
$$
On the other hand, we have that
\begin{align*}
 \|T(x-x_0)\|^2&=\sum_{j=1}^m |\langle x-x_0,m^{-1/2}v_j\rangle|^2\\
 &=m^{-1}\sum_{j=1}^m |\langle x_0-x,v_j\rangle|^2\\
 &> m^{-1}\sum_{j\in J} (a/2)^2\\
 &> m^{-1}(  \gamma^2 m/8) (a/2)^2\\
 &=32^{-1}\gamma^2a^2 
\end{align*}
Thus, by combining our upper bound for $\|T(x-x_0)\|^2$ with our lower bound for $\|T(x-x_0)\|^2$, we have that $32^{-1}\gamma^2a^2<2\vp^2$.  This contradicts that $\vp\leq8^{-1} \gamma a$.
\end{proof}

The following is an adaptation of Lemma \ref{L:signs} to the finite-dimensional setting where $(v_j)_{j=1}^m$ is a frame of $\ell_2^n$.

\begin{lem}\label{L:signs_sG}
Let $(v_j)_{j=1}^m$ be a frame of $\ell_2^n$ with upper frame bound $2$.  Let $f,g\in\ell_2^n$ with $f=\sum_{i=1}^n a_i e_i$ and $g=\sum_{i=1}^n b_i e_i$.  For all $1\leq i\leq n$ let $\vp_i=sign(a_i b_i)$ and $\delta_i=b_i-\vp_i a_i$.  Let $h=\sum_{i=1}^n \vp_i a_i e_i$.  Then,
$$\big\| \big(|\langle f,v_j\rangle| - |\langle h,v_j\rangle|\big)_{j=1}^m\big\|_{\ell_2^m}\leq \big\| \big(|\langle f,v_j\rangle| - |\langle g,v_j\rangle|\big)_{j=1}^m\big\|_{\ell_2^m}+2^{1/2}(\sum_{i=1}^n \delta_i^2)^{1/2}
$$
\end{lem}

\begin{proof}
 We have that
\begin{align*}
 \big\| \big(|\langle f,v_j\rangle| - |\langle g,v_j\rangle|\big)_{j=1}^m\big\|_{\ell_2^m}
&= \big\| \big(|\langle f,v_j\rangle| - |\langle h+\sum_{i=1}^n \delta_i e_i),v_j\rangle|\big)_{j=1}^m\big\|_{\ell_2^m}\\
&\geq \big\| \big(|\langle f,v_j\rangle| - |\langle h,v_j\rangle|\big)_{j=1}^m\big\|_{\ell_2^m}- \big\| \big(\langle \sum_{i=1}^n \delta_i e_i,v_j\rangle\big)_{j=1}^m\big\|_{\ell_2^m}\\
&\geq \big\| \big(|\langle f,v_j\rangle| - |\langle h,v_j\rangle|\big)_{j=1}^m\big\|_{\ell_2^m}-2^{1/2}(\sum_{i=1}^n \delta_i^2)^{1/2}\\
&\hspace{2cm}\big(\textrm{as 2 is an upper frame bound of $(v_j)_{j=1}^m$.})
\end{align*}

\end{proof}

The following is an adaptation of Lemma \ref{L:lower_b} to the finite-dimensional setting.

\begin{lem}\label{L:lower_b_sG}
Let $(v_j)_{j=1}^m$ be a sequence of vectors in $\ell_2^n$. 
Let $x,y\in \ell_2^n$ be vectors such that $|J_{x,y}(a/2)|\geq (\gamma^2/4)m$. Then,
$$2^{-1}a\gamma \min(\|x\|,\|y\|)\leq \big\| \big(|\langle x+y,m^{-1/2}v_j\rangle| - |\langle x-y,m^{-1/2}v_j\rangle|\big)_{j=1}^m\big\|_{\ell_2^m}.
$$
\end{lem}

\begin{proof}
We have that 
\begin{align*}
\big\| \big(|\langle x+y,m^{-1/2}v_j\rangle| - |\langle x-y,m^{-1/2}v_j\rangle|\big)_{j=1}^m\big\|_{\ell_2^m}^2&=m^{-1}\sum_{j=1}^m (2 \min(|\langle x,v_j\rangle|,|\langle y,v_j\rangle|))^2 \\ 
&\geq m^{-1}\sum_{j\in J_{x,y}(a/2)} (2\min(a2^{-1}\|x\|,a2^{-1}\|y\|))^2 \\
&= a^2 m^{-1} |J_{x,y}(a/2)| (\min(\|x\|,\|y\|))^2 \\
&\geq a^2 m^{-1} (4^{-1}\gamma^2 m)(\min(\|x\|,\|y\|))^2 \\
&= 4^{-1}a^2 \gamma^2 (\min(\|x\|,\|y\|))^2 
\end{align*}

\end{proof}

We now state and prove the main theorem of this section.

\begin{thm}\label{T:main_P_sG}
There exists universal constants $c_1,c_2,c_3>0$ such that the following holds.
Let $(y_j)_{j=1}^n$ be an independent sequence of  mean-zero, variance-one, $K$-sub-Gaussian random variables.  Let $(v_j)_{j=1}^m$ be an independent sequence in $\ell_2^n$ with the same distribution as $v=(y_1,y_2,...,y_n)$. If $m\geq c_1 K^{16}\log(2K) n$ then with probability at least $1-\exp\left(-c_2 m \right)$ the frame $(\frac{1}{\sqrt{m}}v_j)_{j=1}^m\cup (e_i)_{i=1}^n$ does $c_3 K^{6}$-stable phase retrieval.   That is, if $T:\ell_2^n\rightarrow\ell_2^{m+n}$ is the analysis operator for the frame $(\frac{1}{\sqrt{m}}v_j)_{j=1}^m\cup (e_i)_{i=1}^n$ then for all $f,g\in\ell_2^{n}$ we have that 

$$
min(\| f-g\|_{\ell_2^n},\| f+g\|_{\ell_2^n}) \leq c_3 K^{6}
\big\| |Tf|-|Tg| \big\|_{\ell_2^{m+n}}. 
$$

\end{thm}

\begin{proof}

Let $a=16^{-1}A_1 K^{-2}$ and $\gamma=a^2$.  By Lemma \ref{L:KL1}, we have that 
$$\textrm{Prob}(|\langle x, v\rangle|\geq a\|x\|_{\ell_2^n})\geq\gamma\qquad\textrm{ for all }x\in \ell_2^n.
$$

Let $\vp=8^{-1} \gamma a$. By Corollary \ref{C:lower_bp_sG}, we have with probability at least 
$1-\exp\big((\log(6\vp^{-1})n-2^{-1}\gamma^4 m \big)$ that
$|J_{x,y}(a)|\geq \gamma^2/4$
for every $(x,y)\in D_\vp$.   Suppose that $m\geq 64C^6K^4n$.   By Corollary \ref{C:frame_sG} we have with probability at least
$1-2\exp(-64^{-1}C^{-6}K^{-4}m)$  that $(m^{-1/2}v)_{j=1}^m$ is a frame of $\ell_2^n$ with upper frame bound $2$.   We substitute $\gamma=8^{-2}K^{-4}$ and $\vp=8^{-4}K^{-6}$ then use a union bound on the probabilities and Theorem \ref{T:all} to conclude that with probability at least 
$$1-\Big(2\exp(-64^{-1}C^{-6}K^{-4}m)+\exp\big((\log(6(8^4 K^6))n-2^{-1}8^{-8}K^{-16} m \big)\Big)$$
that 
 $|J_{x,y}(a/2)|\geq (\gamma^2/4)m$ for all $x,y\in \ell_2^n$ with disjoint support.  As $K\geq 2^{-1/2}$ we have that $\log(2K)>0$.  Hence, there exists constants $c_1,c_2>0$ such that if $m\geq c_1 K^{16}\log(2K) n$ then with probability at least 
 $1-\exp\left(-c_2 m \right)$ we have that $|J_{x,y}(a/2)|\geq (\gamma^2/4)m$ for all $x,y\in \ell_2^n$ with disjoint support and that $(m^{-1/2}v_j)_{j=1}^m$ is a frame of $\ell_2^n$ with upper frame bound $2$.  

 For the remainder of the proof we assume that 
$(m^{-1/2}v_j)_{j=1}^m$ is a frame of $\ell_2^n$ with upper frame bound $2$, and that $|J_{x,y}(a/2)|\geq (\gamma^2/4)m$ for all $x,y\in \ell_2^n$ with disjoint support. Our goal now is to determine a constant $c_3>0$ such that $(m^{-1/2}v_j)_{j=1}^m\cup(e_i)_{i=1}^n$ does $c_3K^6$-stable phase retrieval.

Let $f,g\in \ell_2^n$ with $f=\sum_{j=1}^n a_j e_j$ and $g=\sum_{j=1}^n b_j e_j$.   For all $1\leq j\leq n$ we let $\vp_j= sign(a_jb_j)$ and $\delta_j= b_j  - \vp_j a_j$.  
We let $I=\{j\in\N\,:\, \vp_j=1\}$.  Let $x=\sum_{j\in I}a_j e_j$ and $y=\sum_{j\in I^c} a_j e_j$.  We have that $\sum_{j=1}^n a_j y_j= x+y$ and $\sum_{j=1}^n \vp_j a_j e_j=x-y$.

  Without loss of generality, we may assume that $\|x\|\geq \|y\|$.  
\begin{align*}
(1+(&1+\sqrt{2})^2)^{1/2}\Big(\big\| \big(|\langle f,m^{-1/2}v_j\rangle| - |\langle g,m^{-1/2}v_j\rangle|\big)_{j=1}^m\big\|^2_{\ell_2^m}+\big\| \big(|\langle f,e_i\rangle| - |\langle g,e_i\rangle|\big)_{i=1}^n\big\|^2_{\ell_2^n}\Big)^{1/2}\\
&\geq \big\| \big(|\langle f,m^{-1/2}v_j\rangle| - |\langle g,m^{-1/2}v_j\rangle|\big)_{j=1}^m\big\|_{\ell_2^m}+(1+\sqrt{2})\big\| \big(|\langle f,e_i\rangle| - |\langle g,e_i\rangle|\big)_{i=1}^n\big\|_{\ell_2^n}\\
&= \big\| \big(|\langle f,m^{-1/2}v_j\rangle| - |\langle g,m^{-1/2}v_j\rangle|\big)_{j=1}^m\big\|_{\ell_2^m}+(1+\sqrt{2})\big(\sum_{i=1}^n \delta_i^2\big)^{1/2} \\
 &\geq \big\| \big(|\langle x+y,m^{-1/2}v_j\rangle| - |\langle x-y,m^{-1/2}v_j\rangle|\big)_{j=1}^m\big\|_{\ell_2^m}+\big(\sum_{i=1}^n \delta_i^2\big)^{1/2} \qquad \big(\textrm{by Lemma \ref{L:signs_sG}}\big)\\
 &\geq 2^{-1}a\gamma \| y\|+\big(\sum \delta_i^2\big)^{1/2}\qquad \big(\textrm{by Lemma \ref{L:lower_b_sG}}\big)\\
 &=4^{-1}a\gamma\|(x+y)-(x-y)\|+\big(\sum \delta_i^2\big)^{1/2}\\
&= 4^{-1}a\gamma \left\| \sum a_i e_i - \sum \vp_i a_i e_i\right\| +\big(\sum \delta_i^2\big)^{1/2}\\
&\geq 4^{-1}a\gamma \left\| \sum a_i e_i - \sum (\vp_i a_i+\delta_i) e_i\right\|\hspace{3cm}\textrm{( as $a\gamma\leq1$)}\\
&= 4^{-1}a\gamma \| f- g\|
\end{align*}
This proves that the frame $(\frac{1}{\sqrt{m}}v_j)_{j=1}^m\cup (e_i)_{i=1}^n$ does $(1+(1+2^{1/2})^2)^{1/2}4a^{-1}\gamma^{-1}$-stable phase retrieval.  As $a=8^{-1}K^{-2}$ and $\gamma=8^{-2} K^{-4}$ the frame $(\frac{1}{\sqrt{m}}v_j)_{j=1}^m\cup (e_i)_{i=1}^n$ does $c_3 K^6$-stable phase retrieval where $c_3=(1+(1+2^{1/2})^2)^{1/2}2^{11}$.
\end{proof}

\section{Stability for phase retrieval in finite dimensions}\label{S:5}

In the previous section we proved that if we used sub-Gaussian random variables then with high probability our construction does phase retrieval with stability constant independent of the dimension $n$ when using $m$ on the order of $n$ random vectors.  Without sub-Gaussian assumptions, 
 we prove in Theorem \ref{T:main_P} that with high probability our construction does phase retrieval with stability constant independent of the dimension $n$ when using $m$ on the order of $n\log n$ random vectors.

As we did in Section \ref{S:4} before the statement of Lemma \ref{L:lower_bP_sG}, we now set some notation and give motivation for how we will proceed.  Let $(y_j)_{j=1}^n$ be an independent sequence of  mean-zero, variance-one random variables and let $a,\gamma>0$ be constants such that
$\textrm{Prob}(|y|\geq a\|y\|_{L_2})\geq\gamma$ for all $y\in span(y_j)_{1\leq j\leq n}$.    Let $(v_j)_{j=1}^m$ be an independent sequence of random vectors in $\ell_2^n$ each with the same distribution as $v=(y_1,y_2,...,y_n)$.

For each $x,y,z\in \ell_2^n$ where $x$ and $y$ have disjoint support and $b,s>0$ we let $J_{x,y}^{z}(b,s)$ be the set of all $j$ in $\{1,2,...,m\}$ such that 
\begin{enumerate}
\item[(i)] $|\langle x,v_j\rangle|\geq b\|x\|_{L_2}$ and $|\langle y,v_j\rangle|\geq b\|y\|_{L_2}$,\\
\item[(ii)] $|\langle z, v_j\rangle|\leq s\|z\|_{L_2}$,\\
\item[(iii)] $\|v_j\|\leq 2 \gamma^{-1}  n^{1/2}$.
\end{enumerate}

Condition (i) is the same as \eqref{E:J_sG} in Section \ref{S:4} and guarantees that both $|\langle x,v_j\rangle|$ and $|\langle y,v_j\rangle|$ are relatively large.  Conditions (ii) and (iii) will allow us to prove that if both $|\langle x,v_j\rangle|$ and $|\langle y,v_j\rangle|$ are relatively large and $x',y'\in\ell_2^n$ are close to $x$ and $y$ respectively then both $|\langle x',v_j\rangle|$ and $|\langle y',v_j\rangle|$ are relatively large.  In other words, conditions (ii) and (iii) imply that condition (i) is stable under small perturbations. We did not need conditions $(ii)$ and $(iii)$ in Section \ref{S:4} because there we had with high probability that $(\frac{1}{\sqrt{m}}v_j)_{j=1}^m$ is a frame of $\ell_2^n$ with lower frame bound $1/2$ and upper frame bound $2$.

The following Lemma gives that for all $x,y\in \ell_2^n$ with disjoint support and all $z\in\ell_2^n$ we have with high probability that  $|J^z_{x,y}(a,2\gamma^{-1})|\geq (\gamma^2/4)m$.

\begin{lem}\label{L:lower_bP}
Let $(y_j)_{j=1}^n$ be an orthonormal sequence of independent mean-zero random variables in $L_2([0,1])$ and let $a,\gamma>0$  so that $\textrm{Prob}(|f|\geq a\|f\|_{L_2})\geq\gamma$ for all $f\in span(y_j)_{j=1}^n$.  Then for all $m\in\N$, all $x,y\in \ell_2^n$ with disjoint support, and all $z\in\ell_2^n$ we have that 
$$\textrm{Prob}\Big(|J^z_{x,y}(a,2\gamma^{-1})|\leq 4^{-1}\gamma^2m\Big)\leq e^{-8^{-1} \gamma^4 m}.
$$
\end{lem}

\begin{proof}

As $x$ and $y$ are supported on disjoint coordinates of $\ell_2^n$ and the coordinates of $v$ are independent, we have that
\begin{equation}\label{E:1e}
\textrm{Prob}\Big(|\langle x,v\rangle |\!\geq\! a\|x\|\textrm{ and }|\langle y,v\rangle|\!\geq\! a \|y\|\Big)\!=\! \textrm{Prob}\Big(|\langle x,v\rangle |\!\geq\! a\|x\|\Big)\textrm{Prob}\Big(|\langle y,v\rangle |\!\geq\! a\|y\|\Big)\!\geq\! \gamma^2
\end{equation}

Let $z=(a_j)_{j=1}^n\in\ell_2^n$.  As  $(y_j)_{j=1}^n$ is an ortho-normal sequence we have that
$$
\E |\langle z,v\rangle|^2=\E \big|\sum_{j=1}^n a_j y_j\big|^2=\sum_{j=1}^n |a_j|^2=\|z\|^2.
$$

Using the above equality with Markov's inequality we have that,
\begin{equation}\label{E:2e}
\textrm{Prob}\Big(|\langle z,v\rangle|> 2\gamma^{-1}\|z\|\Big)=\textrm{Prob}\Big(|\langle z,v\rangle|^2> 4\gamma^{-2}\|z\|^2\Big)\leq (\|z\|^2) (4^{-1}\gamma^{2} \|z\|^{-2})= 4^{-1}\gamma^2.
\end{equation}

As $(y_j)_{j=1}^n$ is an orthonormal sequence  we have that
$\E(\|v\|^2)=\E\sum_{j=1}^n |y_j|^2=n.$
    By using Markov's inequality again, 
\begin{equation}\label{E:3e}
\textrm{Prob}(\|v\|\geq 2 \gamma^{-1} n^{1/2})=\textrm{Prob}(\|v\|^2\geq 4 \gamma^{-2} n)\leq \E(\|v\|^2)4^{-1} \gamma^{2}n^{-1}=4^{-1}\gamma^{2}.
\end{equation}

Thus, for each $1\leq j\leq m$ we have that
\begin{align*}
\textrm{Prob}&\Big(j\in J^z_{x,y}(a,2\gamma^{-1})\Big)\\
&\geq \textrm{Prob}\Big(|\langle x,v_j\rangle|\!\geq\! a\|x\|\textrm{ and }|\langle y, v_j\rangle|\!\geq\! a \|y\|\Big)\!-\!\textrm{Prob}\Big(|\langle z,v_j\rangle|\!>\! 2\gamma^{-1}\|z\|\Big)\!-\!\textrm{Prob}\Big(\|v_j\|\!>\! 2\gamma^{-1}n^{1/2}\Big)\\
&\geq \gamma^2-\gamma^2/4-\gamma^2/4\hspace{2cm}\big(\textrm{by \eqref{E:1e}, \eqref{E:2e}, and \eqref{E:3e} }\big)\\
&=\gamma^2/2
\end{align*}
We now consider a sequence $(b_j)_{j=1}^m$ of iid Bernoulli random variables by setting
 $b_j=1$ if $j\in J^z_{x,y}(a,2\gamma^{-1})$ and $b_j=0$ otherwise.  Each $b_j$ has expectation at least $\gamma^2/2$.         By Hoeffding's inequality we have for all $\lambda>0$ that
$$\textrm{Prob}\Big(|J^z_{x,y}(a,2\gamma^{-1})|\leq (\gamma^2/2-\lambda)m\Big)= \textrm{Prob}\left(\sum_{j=1}^m b_j\leq (\gamma^2/2-\lambda)m\right) \leq e^{-2\lambda^2 m}
$$
We now choose $\lambda=\gamma^2/4$ to obtain our desired result.
\end{proof}

\begin{cor}\label{C:lower_bp}
Let $(y_j)_{j=1}^n$ be an orthonormal sequence of independent random variables in $L_2([0,1])$ such that there exists $a,\gamma>0$ so that $\textrm{Prob}(|x|\geq a\|x\|_{L_2})\geq\gamma$  for all $x\in span(y_j)$.
 Let $m\in\N$ and let $(v_j)_{j=1}^m$ be an independent sequence of random vectors in $\ell_2^n$ each with the same distribution as $v=(y_1,y_2,...,y_n)$.
Then with probability at least $1-\exp\left(\log(288a^{-2} \gamma^{-2}  n)n-8^{-1}\gamma^4 m \right)$, we have that
$$|J^z_{x,y}(2^{-1}a,3\gamma^{-1})|\geq (\gamma^2/4)m,$$
for every $x,y,z\in \ell^n_2\setminus\{0\}$  with  $x\in span(e_j)_{j\in I}$ and $y\in span(e_j)_{j\in I^c}$ for some $I\subseteq[n]$.
\end{cor}

\begin{proof}
Let $\vp=4^{-1}a \gamma  n^{-1/2}$.  Note that $0<a\gamma\leq1$ and hence $0<\vp< 1$. For all $I\subseteq [n]$ we may choose a set of unit vectors $Z_I\subseteq span(e_j)_{j\in I}$ which are $\vp$ dense in the unit sphere of $span(e_j)_{j\in I}$ and $|Z_I|\leq (3/\vp)^{|I|}$.   The cardinality of $D:=\cup_{I\subseteq[n]}Z_I\times Z_{I^c}\times Z_{[n]}$ has the following upper bound.
$$|D|\leq \sum_{I\subseteq[n]}  |Z_I| |Z_{I^c}||Z_{[n]}| \leq \sum_{I\subseteq[n]}  \left(\frac{3}{\vp}\right)^{|I|}\left(\frac{3}{\vp}\right)^{|I^c|}\left(\frac{3}{\vp}\right)^{n}=2^n  3^{2n} \vp^{-2n}=18^n \vp^{-2n}.
$$

By using Lemma \ref{L:lower_bP} and a union bound, we have that the probability that 
$|J^z_{x,y}(a,2\gamma^{-1})|\geq (\gamma^2/4)m$ for every $(x,y,z)\in D$ is at least,
\begin{align*}
1-\sum_{(x,y,z)\in D}\textrm{Prob}\Big(|J^z_{x,y}(a,2\gamma^{-1})|< (\gamma^2/4)m\Big)&\geq 1-|D| e^{-8^{-1}\gamma^4 m}\\
&\geq 1-18^n\vp^{-2n} e^{-8^{-1}\gamma^4 m}\\
&=1-\exp\left((\log(18)-\log(\vp^2))n-8^{-1}\gamma^4 m \right)\\
&=1-\exp\left((\log(18)+\log(16a^{-2} \gamma^{-2}  n)n-8^{-1}\gamma^4 m \right)\\
&=1-\exp\left(\log(288a^{-2}\gamma^{-2}n)n-8^{-1}\gamma^4 m \right)
\end{align*}

We now assume that $|J^z_{x,y}(a,2\gamma^{-1})|\geq (\gamma^2/4)m$ for all $(x,y,z)\in D$.  Let  $x_0,y_0,z_0\in \ell^n_2\setminus\{0\}$ with $x_0\in span(e_j)_{j\in I}$ and $y_0\in span(e_j)_{j\in I^c}$ for some $I\subseteq [n]$.  Choose some $(x,y,z)\in D$ with $\|x_0/\|x_0\|-x\|<\vp$, $\|y_0/\|y_0\|-y\|<\vp$, and $\|z_0/\|z_0\|-z\|<\vp$.      Let $j\in J^{z}_{x,y}(a,2\gamma^{-1})$.  We will prove that $j\in J^{z_0}_{x_0,y_0}(a/2,3\gamma^{-1})$ and hence $J^{z}_{x,y}(a,2\gamma^{-1})\subseteq J^{z_0}_{x_0,y_0}(a/2,3\gamma^{-1})$.
We have that 
\begin{align*}
|\langle x_0, v_j\rangle|/\|x_0\|&\geq |\langle x, v_j\rangle| - |\langle x_0/\|x_0\|-x,v_j\rangle|\\
& \geq |\langle x,v_j\rangle|-\|v_j\|\big\|x_0/\|x_0 \|-x\big\|\\
&\geq a-2 \gamma^{-1}  n^{1/2} \vp\\
&=a/2\hspace{2cm}\textrm{( as $\vp=4^{-1}a \gamma  n^{-1/2}$ .)}
\end{align*}
Thus, $|\langle x_0, v_j\rangle|\geq (a/2)\|x_0\|$ and likewise $|\langle y_0, v_j\rangle|\geq (a/2)\|y_0\|$.  This proves that property (i) is satisfied for  $J^{z_0}_{x_0,y_0}(a/2,3\gamma^{-1})$. Note that $\textrm{Prob}(|x|\geq a\|x\|_{L_2})\geq\gamma$ for all $x\in \ell_2^n$ implies that $a\gamma\leq1$. We now check property (ii) using similar inequalities.
\begin{align*}
|\langle z_0,v_j\rangle|/\|z_0\|&\leq |\langle z, v_j\rangle| +|\langle z_0/\|z_0\|-z,v_j\rangle|\\
&\leq |\langle z,v_j\rangle|+\|v_j\|\big\|z_0/\|z_0 \|-z\big\|\\
&\leq 2\gamma^{-1}+2\gamma^{-1}  n^{1/2} \vp\\
&=2\gamma^{-1}+a/2\hspace{2cm}\textrm{( as $\vp=4^{-1}a \gamma  n^{-1/2}$ )}\\
&< 3\gamma^{-1}\hspace{3cm}\textrm{( as $a\gamma\leq 1$.)}
\end{align*}
Thus, property (ii) is satisfied for $ J^{z_0}_{x_0,y_0}(a/2,3\gamma^{-1})$.  Property (iii) is the same for both $J^{z_0}_{x_0,y_0}(a/2,3\gamma^{-1})$ and $J^{z_0}_{x_0,y_0}(a,2\gamma^{-1})$, hence all three properties are satisfied and $j\in J^{z_0}_{x_0,y_0}(a/2,3\gamma^{-1})$.  This proves that $J^{z}_{x,y}(a,2\gamma^{-1})\subseteq J^{z_0}_{x_0,y_0}(a/2,3\gamma^{-1})$.  Hence, $|J^{z_0}_{x_0,y_0}(a/2,3\gamma^{-1})|\geq (\gamma^2/4)m$.

\end{proof}

We now state and prove the main theorem for this section.

\begin{thm}\label{T:main_P}
Let $a,\gamma>0$.  Then there exists constants $k_1,k_2>0$ which depend only on the values $a$ and $\gamma$ such that for all $n\in\N$ and $m\geq k_1 n\log(n)$ the following holds.
Suppose that $(y_j)_{j=1}^n$ is a sequence of  mean-zero, variance-one, independent random variables such that  $\textrm{Prob}(|x|\geq a\|x\|_{L_2})\geq\gamma$ for all $x\in span(y_j)$.
Let $(v_j)_{j=1}^m$ be an independent sequence of random vectors in $\ell_2^n$ each with the same distribution as $v=(y_1,y_2,...,y_n)$. Then with probability at least $1-\exp\left(-k_2 m \right)$ the frame $(\frac{1}{\sqrt{m}}v_j)_{j=1}^m\cup (e_i)_{i=1}^n$ does $(12\gamma^{-1}a^{-1}+1)$-stable phase retrieval.  That is, if $T:\ell_2^n\rightarrow\ell_2^{m+n}$ is the analysis operator for the frame $(\frac{1}{\sqrt{m}}v_j)_{j=1}^m\cup (e_i)_{i=1}^n$ then for all $f,g\in\ell_2^{n}$ we have that 

$$
min(\| f-g\|_{\ell_2^n},\| f+g\|_{\ell_2^n}) \leq(12a^{-1}\gamma^{-1} +1)
\| |Tf|-|Tg| \|_{\ell_2^{m+n}}. 
$$

\end{thm}

\begin{proof}
By Corollary \ref{C:lower_bp}, we have with probability at least  $1-\exp\left(\log(288a^2 \gamma^{-2}  n)n-8^{-1}\gamma^4 m \right)$ that
\begin{equation}\label{E:lower_bP}
|J^z_{x,y}(a/2,3\gamma^{-1})|\geq (\gamma^2/4)m,
\end{equation}
for every $x,y,z\in \ell_2^n\setminus\{0\}$  with  $x\in span(y_j)_{j\in I}$ and $y\in span(y_j)_{j\in I^c}$ for some $I\subseteq[n]$.  Thus, we may choose $k_1,k_2>0$ to depend only on $a$ and $\gamma$ so that if $m\geq k_1 n\log(n)$ then \eqref{E:lower_bP} holds with probability at least $1-\exp\left(-k_2 m \right)$.
We now assume that \eqref{E:lower_bP} holds and will prove that the frame $(\frac{1}{\sqrt{m}}v_j)_{j=1}^m\cup (e_i)_{i=1}^n$ does $(12\gamma^{-1}a^{-1}+1)$-stable phase retrieval.

Let $f,g\in \ell_2^n$ with $f=\sum_{j=1}^n a_j e_j$ and $g=\sum_{j=1}^n b_j e_j$.   For all $1\leq j\leq n$ we let $\vp_j= sign(a_jb_j)$ and $\delta_j= b_j  - \vp_j a_j$.  
We let $I=\{j\in[n]:\, \vp_j=1\}$.  Let $x=\sum_{j\in I}a_j e_j$ and $y=\sum_{j\in I^c} a_j e_j$.  We have that $x+y=\sum_{j=1}^n a_j e_j$ and $x-y=\sum_{j=1}^n \vp_j a_j e_j$. Let $z=\sum_{j=1}^n \delta_j e_j$ and note that $|\delta_j|=||a_j|-|b_j||$.

    Without loss of generality, we may assume that $\|x\|\geq \|y\|$.  Let $T:\ell_2^n\rightarrow\ell_2^{m+n}$ be the analysis operator for the frame $(\frac{1}{\sqrt{m}}v_j)_{j=1}^m\cup (e_i)_{i=1}^n$
    We first consider the case that $\|y\|< 6\gamma^{-1}a^{-1}\|z\|$.  
   As $(\frac{1}{\sqrt{m}}v_j)_{j=1}^m\cup (e_i)_{i=1}^n$ includes the unit vector basis, we have that,
$$
 \big\| |Tf|-|Tg| \big\|^2_{\ell^{m+n}_2}\geq \sum_{j=1}^n \big||a_j|-|b_j|\big|=\|z\|^2 
$$   
We now estimate $\|f-g\|$ by
\begin{align*}
    \|f-g\|&=\|(x+y)-(x-y+z)\|\\
    &\leq 2\|y\|+\|z\|\\
    &\leq (12 \gamma^{-1} a^{-1}+1)\|z\|\hspace{2cm}\left(\textrm{as $\|y\|< 6\gamma^{-1}a^{-1}\|z\|$}\right)\\
    &\leq (12 \gamma^{-1} a^{-1}+1)\big\| |Tf|-|Tg| \big\|_{\ell^{m+n}_2}
\end{align*}

    Hence,  in the case that $\|y\|< 6\gamma^{-1}a^{-1}\|z\|$ we have that $\|f-g\|\leq (12 \gamma^{-1} a^{-1}+1)\| |Tf|-|Tg| \|_{\ell^{m+n}_2}$.  We now assume that $\|y\|\geq 6\gamma^{-1}a^{-1}\|z\|$.

    Let $T:\ell_2^n\rightarrow\ell_2^{m+n}$ be the analysis operator for the frame $(\frac{1}{\sqrt{m}}v_j)_{j=1}^m\cup (e_i)_{i=1}^n$, and let $T_0:\ell_2^n\rightarrow\ell_2(J_{x,y}^z(\frac{a}{2},\frac{3}{\gamma}))$ be the analysis operator for $(\frac{1}{\sqrt{m}}v_j)_{j\in J_{x,y}^z(\frac{a}{2},\frac{3}{\gamma})}$.  We now estimate a lower bound on $\| |Tf|-|Tg| \|_{\ell^{m+n}_2}$.

\begin{align*}
\sqrt{2}\| |Tf|&-|Tg| \|_{\ell^{m+n}_2}\geq  \left\| |T_0(x+y)| - |T_0(x-y+z)|\right\|_{\ell_2(J_{x,y}^z(\frac{a}{2},\frac{3}{\gamma}))}+\|z\|\\
&\geq \left\| |T_0(x+y)| - |T_0(x-y)|\right\|_{\ell_2(J_{x,y}^z(\frac{a}{2},\frac{3}{\gamma}))}-\left\| T_0(z)\right\|_{\ell_2(J_{x,y}^z(\frac{a}{2},\frac{3}{\gamma}))}+\|z\|\\
&=\Big(\sum_{j\in J_{x,y}^z(\frac{a}{2},\frac{3}{\gamma})}m^{-1}\big(|\langle x+y, v_j\rangle| - |\langle x-y,v_j\rangle|\big)^2\Big)^{1/2} -\Big(\sum_{j\in J_{x,y}^z(\frac{a}{2},\frac{3}{\gamma})}m^{-1}|\langle z,v_j\rangle|^2\Big)^{1/2}+\|z\|\\
&=\Big(\sum_{j\in J_{x,y}^z(\frac{a}{2},\frac{3}{\gamma})} 4m^{-1}\min(|\langle x, v_j\rangle|^2, |\langle y,v_j\rangle|^2)\Big)^{1/2}- \Big(\sum_{j\in J_{x,y}^z(\frac{a}{2},\frac{3}{\gamma})} m^{-1}|\langle  z,v_j\rangle|^2\Big)^{1/2}+\|z\|\\
&\geq |{J_{x,y}^z(\tfrac{a}{2},\tfrac{3}{\gamma})}|^{1/2}m^{-1/2}a \|y\|-|{J_{x,y}^z(\tfrac{a}{2},\tfrac{3}{\gamma})}|^{1/2} m^{-1/2} 3\gamma^{-1}\|z\|+\|z\|\\
&\geq |{J_{x,y}^z(\tfrac{a}{2},\tfrac{3}{\gamma})}|^{1/2}2^{-1}m^{-1/2}a \|y\|+\|z\|\hspace{1.5cm}\left(\textrm{as $\|y\|\geq 6\gamma^{-1}a^{-1}\|z\|$}\right)\\
&\geq 4^{-1}\gamma a\|y\|+\|z\|\hspace{1.5cm}\left(\textrm{ by \eqref{E:lower_bP}}\right)\\
&\geq 8^{-1}\gamma a \|(x+y)-(x-y+z)\|\\
&=8^{-1}\gamma a \|f-g\|
\end{align*}

Thus, in the case that $\|y\|\geq 6\gamma^{-1}a^{-1}\|z\|$ we have that $\|f-g\|\leq 8\sqrt{2}\gamma^{-1}a^{-1}\||Tf|-|Tg|\|$.  After comparing this with the other case, we have in general that $\|f-g\|\leq (12\gamma^{-1}a^{-1}+1)\||Tf|-|Tg|\|$.
Thus, the frame $(\frac{1}{\sqrt{m}}v_j)_{j=1}^m\cup (e_i)_{i=1}^n$ of $\ell_2^n$ does $(12\gamma^{-1}a^{-1}+1)$-stable phase retrieval.

\end{proof}


\end{document}